\documentclass[10pt,letterpaper]{amsart}
\usepackage{amssymb,graphicx}
\usepackage[T1]{fontenc} \usepackage{mathptmx} 
\newcommand{\del}[1]{\frac{\partial}{\partial #1}}
\newcommand{\indel}[1]{\partial/\partial #1}

\newtheorem{theorem}{Theorem}
\newtheorem{lemma}[theorem]{Lemma}
\newtheorem{proposition}[theorem]{Proposition}

\theoremstyle{remark}
\newtheorem{remark}[theorem]{Remark}
\newtheorem{example}[theorem]{Example}
\newtheorem{definition}[theorem]{Definition}

\title{Quasihomogeneous analytic  affine connections on surfaces}
\author{Sorin Dumitrescu}\address{Universit\'e Nice-Sophia Antipolis, Laboratoire J.-A. Dieudonn\'e, UMR 7351 CNRS, Parc Valrose, 06108 Nice Cedex 2, France} \email{dumitres@unice.fr}
\author{Adolfo Guillot}\address{Instituto de Matem\'aticas, Unidad Cuernavaca, Universidad Nacional Aut\'onoma de M\'exico,  A.P.~273-3 Admon.~3, Cuernavaca, Morelos, 62251, Mexico} \email{adolfo.guillot@im.unam.mx}

\thanks{The first author  was  partially supported by  the ANR Grant 08-JCJC-0130-01. The second author was partially supported by CONACyT-Mexico  grant 167594 and enjoyed the hospitality of the Universit\'e  Nice-Sophia Antipolis while preparing this work.}


\begin{document}
\begin{abstract} 
We classify torsion-free real-analytic affine connections on  compact oriented  real-analytic surfaces which are locally homogeneous on a nontrivial open set, without being locally homogeneous on all of the surface. In particular, we prove that such connections exist. This classification relies in a local result that classifies germs of torsion-free real-analytic affine connections on a neighborhood of the origin in the plane  which are quasihomogeneous, in the sense that they are locally homogeneous on an open set containing the origin in its closure, but not locally homogeneous in the neighborhood of the origin. 
\end{abstract}

\maketitle

\section{Introduction}

This article deals with the classification of torsion-free real-analytic affine connections on compact orientable surfaces which are \emph{locally homogeneous} on a nontrivial open set, but not on all of the surface. They are here called \emph{quasihomogeneous}. Our main results are  stated  in  Theorem~\ref{thm-local}  (the local classification) and in Theorem~\ref{thm-global} (the global classification). In particular, we show that such (strictly) quasihomogeneous connections exist.


The study of \emph{locally homogeneous} geometric structures is a classical subject in differential geometry. It has its roots in the seminal work of Lie~\cite{Lie} and took shape through Ehresmann's  article~\cite{ehresmann}. In the Riemannian setting, these locally homogeneous spaces are the context of Thurston's 3-dimensional geometrization program~\cite{thurston}. Locally homogeneous torsion-free connections on surfaces were studied by Opozda in~\cite{Opozda}. In~\cite{BMM}, answering a question of  Lie, the authors classified Riemannian metrics on surfaces whose underlying Levi-Civita connections are projectively locally homogeneous. After the locally homogeneous geometric structures, the most symmetric geometric structures are those that are \emph{quasihomogeneous}, this is, those that are locally homogeneous in an open and dense set. These quasihomogeneous connections are the object of our article.

Our results are also motivated by Gromov's \emph{open-dense orbit theorem}~\cite{DG,Gro} (see also~\cite{Benoist2,CQ,Feres}). Gromov's result asserts that, if the pseudogroup of local automorphisms of  a \emph{rigid geometric structure} (an analytic connection, for example) acts with a dense orbit, then this orbit is open.  In this case, the geometric structure is locally homogeneous on a open dense set. Gromov's theorem says nothing about this maximal open and dense set of local homogeneity. In many interesting geometric situations, it may be all of the (connected)  manifold. This was proved, for instance,  for Anosov flows with differentiable stable and unstable foliations and  transverse contact  structure~\cite{BFL} and for three-dimensional compact Lorentz manifolds admitting a nonproper one-parameter group of automorphisms~\cite{Zeghib}.

In~\cite{BF}, the authors deal with this question  and their results indicate ways in which some rigid geometric structures cannot degenerate off the open dense set. In~\cite{Fisher-zimmer}, Fisher conjectured that the maximal open set of local homogeneity is all of the manifold as soon as the following three conditions are fulfilled: the automorphism group of the manifold acts with a  dense orbit, the geometric structure is a $G$-structure  (meaning that it is locally homogeneous at the first order) and the manifold is compact.

Surprisingly, for some specific geometric structures, if the subset of local homogeneity is not empty, a maximal open set of local homogeneity is all of the (connected) manifold,  \emph{even if we drop the assumption  that  the automorphism group acts with a dense orbit}. This  is known to be true in the Riemannian setting~\cite{Tri}, as a consequence of the fact that all scalar invariants are constant (see also section 3 in~\cite{sorin-locally}). This was also proved for three dimensional real-analytic Lorentz metrics~\cite{Dumitrescu} and, in higher dimension, for complete  real-analytic Lorentz metrics having semisimple Killing Lie algebra~\cite{Melnick}. 
In~\cite{sorin-locally}, the first author proved that a real-analytic unimodular affine connection on a real-analytic compact surface which is locally homogeneous on a nontrivial open set is locally homogeneous on all of the surface, and asks about the extent to which the unimodularity hypothesis is necessary.

Motivated by these results,  Theorem~\ref{thm-global} constructs and characterizes torsion-free real-analytic affine connections on compact surfaces which are quasihomogeneous  (but not homogeneous). 

The main local ingredient of  Theorem~\ref{thm-global}  is the  following classification of  germs of torsion-free real-analytic affine connections which are quasihomogeneous. It is, in some sense, the quasihomogeneous analogue to the local results in~\cite{koo}.

\begin{theorem}\label{thm-local} Let~$\nabla$ be a torsion-free  real-analytic affine connection in a neighborhood of the origin in~$\mathbf{R}^2$. Suppose that the maximal  open set  where  $\mathfrak{K}(\nabla)$, the Lie algebra of Killing vector fields of~$\nabla$, is transitive, contains the origin in its closure,  but does not contain the origin. Then, up to an analytic change of coordinates, the germ of~$\nabla$ at the origin is one of the following:

\begin{description}

\item[Type~$\mathrm{I}(n)$, $n\in\frac{1}{2}\mathbf{Z}$, $n \geq \frac{1}{2}$] The germ at~$(0,0)$ of
\begin{equation*}
\nabla_{\del{x}}\del{x}= 0,\;
\nabla_{\del{x}}\del{y}= -\gamma x^{n}\del{x},\;
 \nabla_{\del{y}}\del{y}=-{\textstyle \frac{1}{n}}\epsilon x^{2n+1}\del{x}-\phi x^{n}\del{y},
 \end{equation*}
with $\phi=0$ and~$\gamma=0$, if~$n\notin\mathbf{Z}$ and $(n,\phi,\epsilon)\neq (1,-\gamma,-\gamma^2)$. For these, $\mathfrak{K}(\nabla)=\langle x\indel{x}-ny\indel{y},\indel{y}\rangle$. 

\item[Type~$\mathrm{II}(n)$, $n\in\frac{1}{2}\mathbf{Z}$, $n \geq \frac{5}{2}$] The germ at~$(0,0)$---Type~$\mathrm{II}^0(n)$---or the germ at~$(0,1)$---Type~$\mathrm{II}^1(n)$---of 
\begin{multline*}
\nabla_{\del{x}}\del{x}=\left(-{\textstyle \frac{1}{n}}\epsilon x^{2n-3}y^2+2\gamma x^{n-2}y\right)\del{x}+\left(-{\textstyle \frac{1}{n}}\epsilon x^{2n-4}y^3+[2\gamma-\phi]x^{n-3}y^2\right)\del{y},\\
 \nabla_{\del{x}}\del{y}=\left({\textstyle \frac{1}{n}}\epsilon x^{2n-2}y-\gamma x^{n-1}\right)\del{x}+\left({\textstyle \frac{1}{n}}\epsilon x^{2n-3}y^2+[\phi-\gamma]x^{n-2}y\right)\del{y},\\
 \nabla_{\del{y}}\del{y}=-{\textstyle \frac{1}{n}}\epsilon x^{2n-1}\del{x}-\left({\textstyle \frac{1}{n}}\epsilon x^{2n-2}y+\phi x^{n-1}\right)\del{y},
\end{multline*}
with $\phi=0$ and~$\gamma=0$, if~$n\notin\mathbf{Z}$. For these, $\mathfrak{K}(\nabla)=\langle x\indel{x}+(1-n)y\indel{y},x\indel{y}\rangle $.  
\item[Type~$\mathrm{III}$] The germ at~$(0,0)$ of
\begin{multline*}
\nabla_{\del{x}}\del{x}=\left(-\textstyle{\frac{1}{2}}\epsilon xy^2+2\gamma y\right)\del{x}-{\textstyle\frac{1}{2}}\epsilon y^3\del{y},\\
 \nabla_{\del{x}}\del{y}=\left(\textstyle{\frac{1}{2}}\epsilon x^{2}y-\gamma x\right)\del{x}+\left(\textstyle{\frac{1}{2}}\epsilon xy^2+\gamma  y\right)\del{y},\\
 \nabla_{\del{y}}\del{y}=-{\textstyle\frac{1}{2}}\epsilon x^{3}\del{x}-\left(\textstyle{\frac{1}{2}}\epsilon x^{2}y+2\gamma x \right)\del{y},
\end{multline*}
for which~$\mathfrak{K}(\nabla)$ is the Lie algebra of divergence-free linear vector fields in $\mathbf{R}^2$.
\end{description}
In Types~$\mathrm{I}$ and~$\mathrm{II}$, $(\gamma,\phi,\epsilon) \in\mathbf{R}^3 \setminus \{(0,0,0) \}$ and the connection with parameters~$(\gamma,\phi,\epsilon)$ is equivalent to the one with parameters~$(\mu\gamma,\mu\phi,\mu^2\epsilon)$, for~$\mu>0$ in the case of Type~$\mathrm{II}^1$ and, for~$\mu\in\mathbf{R}^*$, in  the other cases. In Type~$\mathrm{III}$, $(\gamma,\epsilon)\in\mathbf{R}^2 \setminus \{ (0,0)\}$ and the connection with parameters~$(\gamma,\epsilon)$ is equivalent to the one with parameters~$(\mu\gamma,\mu^2\epsilon)$ for~$\mu\in\mathbf{R}^*$. Apart from this, all the above connections are inequivalent.

\end{theorem}

\begin{remark}
For the connections of Type~$\mathrm{III}$, $\mathfrak{K}(\nabla)\approx\mathfrak{sl}(2,\mathbf{R})$ and there is one two-dimensional orbit of the Killing algebra (the complement of the origin). In the other cases, $\mathfrak{K}(\nabla)\approx\mathfrak{aff}(\mathbf{R})$, the Lie algebra of the affine group of the real line, and the components of the complement of the geodesic~$\{x=0\}$ in $\mathbf{R}^2$  are the two-dimensional orbits of the Killing algebra. In particular, all these germs admit nontrivial open sets on which $\nabla$ is locally isomorphic to a translation invariant connexion on the connected component
of the affine group of the real line. \emph{The closed set where $\nabla$ is not locally homogeneous is either a geodesic, or a point.} Moreover, in all cases, every vector field in the  Killing algebra $\mathfrak{K}(\nabla)$ is an affine one and thus these Killing Lie algebras also preserve a flat torsion-free affine connection.
\end{remark}

Our global result is the following one:

\begin{theorem}\label{thm-global}  \begin{enumerate}\item\label{main-1st} For integers~$n_1,n_2$, with $n_2\geq n_1\geq 2$, there exists a unique (up to automorphism) real-analytic torsion-free  affine connection~$\nabla_{n_1,n_2}$ on~$\mathbf{R}^2$ such that
\begin{enumerate}
\item $\nabla_{n_1,n_2}$ is locally homogeneous on a  nontrivial open set, but not on all of $\mathbf{R}^2$. For $i=1,2$, there exists a point~$p_i\in\mathbf{R}^2$ such that~$\nabla_{n_1,n_2}$ is, in a neighborhood of~$p_i$, given by a normal form of type~$\mathrm{II}^1(n_i)$, if~$n_i$ is odd and by a normal form of type~$\mathrm{I}(n_i)$, if~$n_i$ is even (in particular, the Killing Lie algebra of~$\nabla_{n_1,n_2}$ is isomorphic to that of the affine group of the real line). 
\item There exist groups of automorphisms  of~$\nabla_{n_1,n_2}$ acting freely, properly discontinuously and cocompactly on~$\mathbf{R}^2$.
\end{enumerate}
\item\label{main-rec}Let~$S$ be a compact orientable analytic surface endowed with a real-analytic torsion-free affine connection that is locally homogeneous on some nontrivial open set, but not on all of $S$. Then~$(S,\nabla)$ is isomorphic to a quotient of~$(\mathbf{R}^2,\nabla_{n_1,n_2})$.

\item\label{main-mod} The moduli space of compact quotients of~$(\mathbf{R}^2,\nabla_{n_1,n_2})$ is~$\Xi=\mathbf{N}\times\mathbf{R}\times\mathbf{R}/\mathbf{Z}$. Every compact quotient of~$(\mathbf{R}^2,\nabla_{n_1,n_2})$  is a torus. For the torus $T$ corresponding to~$(k,\tau,\theta)\in\Xi$ we have:
\begin{enumerate}
\item The open set of local homogeneity is dense and is a union of~$2k$ (if~$n_1\neq n_2$) or~$k$ (if~$n_1=n_2$) cylinders bounded by  simple closed geodesics.
\item There exists a globally defined Killing field~$A$ on~$T$, unique up to multiplication by a constant. When normalized  such that there exists a Killing vector field~$B$ defined in some open subset such that~$[A,B]=B$,  $A$ is periodic with period~$\tau$. 
\item If~$\gamma_1$ and~$\gamma_2$ are generators of the fundamental group of~$T$ and~$\gamma_1$ is homotopic to an orbit of~$A$, then the analytic continuations of~$B$ along~$\gamma_1$ and~$\gamma_2$ are, respectively, $e^\tau B$ and~$e^{(\theta+k)\tau}B$.
\end{enumerate}
\end{enumerate}

\end{theorem}

\begin{example} Let us describe explicitly the quotients of~$\nabla_{2,2}$ with parameters~$(1,\tau,[\theta])\in\Xi$. This will also give a self-contained proof of the fact that tori admit quasihomogeneous connections. Consider, in~$\mathbf{R}^2$, the torsion-free affine  connection~$\nabla$ such that
$$\nabla_{\del{x}}\del{x}=0,\; \nabla_{\del{x}}\del{y}=\frac{3}{4}x^2\del{x},\; \nabla_{\del{y}}\del{y}=\frac{3}{8}x^5\del{x}+\frac{3}{4}x^2\del{y}$$
(it is a connection of type~$\mathrm{I}$ for $n=2$ and $\epsilon= \phi=\gamma=-\frac{3}{4}$). Remark that $B=\indel{y}$ is a Killing vector field for $\nabla$. Consider the commuting meromorphic vector fields
$$A=\frac{1}{2}x\del{x}-y\del{y},\;Z=\frac{1}{2}x\del{x}+x^{-2}\del{y}.$$
If we let~$h(x,y)=x^2y$, we have
\begin{equation}\label{invexo}\nabla_A Z=\frac{3}{4}A-Z,\; \nabla_Z Z=-\frac{1}{4}Z,\;\nabla_A A=\frac{3}{4}(h^2+2h)A+\left(\frac{1}{4}-\frac{3}{8}[h^2+2h]\right)Z.\end{equation}

Notice that, since~$[A,Z]=0$ and~$h$ is a first integral of~$A$,  the vector field $A$ is a Killing field of~$\nabla$. The Lie algebra of Killing vector fields of~$\nabla$ contains the subalgebra generated by~$A$ and by~$B$. The rank of this Lie algebra of vector fields is one in~$\{x=0\}$ and two in its complement.  A direct computation  shows that the curvature tensor of $\nabla$ vanishes exactly on $\{x=0\}$.  The connection is thus locally homogeneous in the half planes $\{x<0\}$ and $\{x >0\}$, but not on all of $\mathbf{R}^2$.  

Consider the orientation-preserving birational involution~$\sigma(x,y)=(-x,-y-2x^{-2})$. It preserves the vector fields~$A$ and~$Z$. Moreover, $h\circ \sigma=-2-h$ and hence~$(h^2+2h)\circ \sigma=h^2+2h$. This implies, by~(\ref{invexo}), that~$\nabla$ is preserved by~$\sigma$.   Let~$\Omega=\{(x,y); y<0, h>-2\}$, $U^+=\Omega\cap\{x>0\}$, $U^-=\Omega\cap\{x<0\}$. Notice that~$\sigma|_{U^+}:U^+\to U^-$ is an analytic diffeomorphism. 
Let~$\phi:\Omega\to \Omega$ be the diffeomorphism generated by the flow of~$A$ in time~$\tau$. The diffeomorphism~$\phi$ preserves~$\nabla$ and commutes with~$\sigma$. The quotient of~$\Omega$ under the group generated by~$\phi$ is a cylinder containing a simple closed curve coming from~$\{x=0\}$, whose complement is the union of the two cylinders~$U^+/\langle \phi \rangle$ and~$U^-/\langle \phi \rangle$. Let~$K:U^+/\langle \phi \rangle\to U^-/\langle \phi \rangle$ be given by, first, the restriction to~$U^+/\langle \phi \rangle$ of the flow of~$A$ in time~$\theta\tau$ and then composing with~$\sigma$ (notice that adding an integer to~$\theta$ yields the same result). By identifying~$U^+/\langle \phi \rangle$ and~$U^-/\langle \phi \rangle$ (as open subsets of~$\Omega/\langle \phi \rangle$), via~$K$, we obtain a torus~$S$, naturally endowed with a connection~$\nabla_s$, coming  from~$\nabla$, a globally defined Killing vector field for~$\nabla_s$, induced by~$A$, and a multivalued one, induced by~$B$. There is one simple closed curve in~$S$ coming from~$\{x=0\}/\langle \phi \rangle$. The rank of the Killing algebra of vector fields of~$\nabla_s$ is one along this curve and two in the complement: the connection is not locally homogeneous everywhere (since the curvature tensor vanishes exactly on $\{x=0\}$), but is locally homogeneous in a dense open subset.\end{example}

Note that in the previous example the connected component of the automorphism group of $\nabla_{s}$ is the flow generated by~$A$, all of whose orbits are closed. The proof of Theorem~\ref{thm-global} shows that, in general, the automorphism group  of a quasihomogeneous connection is, up to a finite group, the flow of a Killing field, all of  whose orbits are simple closed curves. Hence, the automorphism group doesn't admit dense orbits (our quasihomogeneous connections are not  counter-examples to Fisher's conjecture).\\

Let us describe the structure of the article and indicate briefly our methods. In section~\ref{Killing fields} we recall the background and the basic definitions. In particular, we recall Nomizu's result about extension of local Killing fields in the real-analytic setting.

The idea for the proof of Theorem~\ref{thm-local} is the following. Lemma~\ref{dim2} proves that for a germ of quasihomogeneous real-analytic affine connection,  we can always  find a two-dimensional  subalgebra  of  the Killing algebra
which is transitive on a nontrivial open set, but not at the origin (Lemma~\ref{dim2}). Consequently, there exists a nontrivial open set where the connection is locally isomorphic either  to
a translation-invariant connection on $\mathbf{R}^2$, or to a left-invariant connection on the affine group. The Abelian case will be dealt with in Section~\ref{commuKillsub}
(Prop.~\ref{propcom}): we show that a quasihomogeneous  connection cannot be locally isomorphic (on a nontrivial open set)  to a translation invariant connection on $\mathbf{R}^2$ (without being  locally homogeneous everywhere). In order to deal with the affine case, we will have to study the invariant affine connections and their Killing algebras in the affine group~(Section~\ref{affine group}) in order to finish, via the same method, the proof of Theorem~\ref{thm-local} (Prop.~\ref{propaffine}, Section~\ref{Quaconnaffine}). The method consists in considering  normal forms at the origin of left invariant vector fields  $X,Y$ on the affine group, with respect to which  we compute Christoffel coefficients  (in general $X$ and $Y$ are meromorphic at the origin).

In order to prove Theorem~\ref{thm-global}, we begin, in Section~\ref{sec:const}, to construct the connections~$\nabla_{n_1,n_2}$ on~$\mathbf{R}^2$ that appear in item~(\ref{main-1st}) of the statement of the Theorem. We then study the symmetries of these connections and prove item~(\ref{main-mod}). Section~\ref{Quasi-compact-surfaces} proves item~(\ref{main-rec}) of Theorem~\ref{thm-global}. The proof depends on all the previous material and  uses~$(G,X)$-structures and their relation to locally homogeneous connections as discussed in Section~\ref{Killing fields}.

Finally, in Section~\ref{sl2chapter} we study geodesically complete affine connections on $\mathbf{R}^n$ which are invariant under the special linear group.

\section{Connections and their symmetries}    \label{Killing fields}

Recall that an \emph{affine connection} $\nabla$  is an operator defined on a manifold $M$ which enables one to consider the derivatives of a smooth vector field  $Y$ with respect to a smooth vector field  $X$ in order to obtain a smooth vector field  $\nabla_{X}Y$ such that, for a function~$f$, $\nabla_{fX}Y=f\nabla_X Y$ and  $\nabla_{X}fY=(X \cdot f)Y+f\nabla_{X}Y$ (it is tensorial in~$X$ and satisfies Leibniz's rule in~$Y$).  The reader will find several equivalent definitions and many details on affine connections in~\cite{kobayashi-nomizu}.

In local coordinates  $(x_{1},x_{2}, \ldots, x_{n})$,  we have  $$\nabla_{\frac{\partial}{\partial x_{i}}}\frac{\partial}{\partial x_{j}}=\Gamma_{ij}^1 \frac{\partial}{\partial x_{1}} +\Gamma_{ij}^2 \frac{\partial}{\partial x_{2}} + \ldots +  \Gamma_{ij}^n \frac{\partial}{\partial x_{n}}$$ for some smooth functions
$\Gamma_{ij}^k$.  These functions, which determine~$\nabla$ locally, are called \emph{Christoffel coefficients} or \emph{Christoffel symbols} of the connection with respect to the coordinates $(x_{1},x_{2}, \ldots, x_{n})$. 
In a real-analytic manifold, $\nabla$ is called \emph{real-analytic} if, for any real-analytic local coordinates $(x_{1},x_{2}, \ldots, x_{n})$, the corresponding Christoffel coefficients $\Gamma_{ij}^k$ are real-analytic functions.

There are two tensors attached to a connection. The \emph{curvature} tensor $$R(X,Y)Z=\nabla_{X} \nabla_{Y}Z-\nabla_{Y}\nabla_{X}Z -\nabla_{\lbrack X, Y \rbrack} Z$$ and the \emph{torsion tensor}
$T(X,Y)=\nabla_XY-\nabla_YX-[X,Y]$. A connection $\nabla$ is said to be \emph{torsion-free} or \emph{symmetric} if the torsion tensor vanishes identically. \emph{All the connections considered in this article are symmetric}. A connection is said to be \emph{flat} if the curvature tensor vanishes. By a classical result  due to \'E.~Cartan,  if a symmetric connection is flat then it is locally isomorphic to the \emph{standard} connection (the one with vanishing Christoffel symbols) on $\mathbf{R}^n$~\cite{Sharpe, Wolf}.

\subsection{Killing fields}
Recall the following classical definition of an infinitesimal symmetry of a connection.

\begin{definition}  A (local) Killing field of  a connection $\nabla$ is a (local) vector field  whose (local) flow  preserves~$\nabla$.
\end{definition}

The terminology comes, by extension, from the Riemannian setting, even if here there is no invariant metric for a Killing field. The Killing vector fields of a connection form naturally a Lie algebra under the Lie bracket. For a connection~$\nabla$, we will denote by~$\mathfrak{K}(\nabla)$ the Lie algebra of its Killing vector fields. For example, we have the following result, whose proof we omit.
\begin{lemma}\label{killingcentralizer} Let~$\nabla$ be a connection in the manifold~$M$ and~$X\in\mathfrak{K}(\nabla)$. If~$Z_1$ and~$Z_2$ commute with~$X$, so does~$\nabla_{Z_1}Z_2$.
\end{lemma}

Following~\cite[Ch.~VI, Prop.~2.2]{kobayashi-nomizu}, a (local) vector field~$X$ is a Killing field of the connection~$\nabla$ if for every pair of (local) vectors fields~$Y$ and~$Z$ we have
\begin{equation}\label{equationkilling}[X,\nabla_Y Z]-\nabla_Y[X,Z]=\nabla_{[X,Y]}Z.\end{equation}

In local coordinates, if
\begin{equation}\label{genericcon}\nabla_{\del{x}}\del{x}=A\del{x}+B\del{y},\; \nabla_{\del{x}}\del{y}=C\del{x}+D\del{y}, \; \nabla_{\del{y}}\del{y}=E\del{x}+F\del{y},\end{equation}
the vector field~$X=a\indel{x}+b\indel{y}$ will be a Killing vector field of~$\nabla$ if and only if its coefficients satisfy the following system of linear partial differential equations:
\begin{eqnarray}
\label{primera} 0 & = & a_{xx}+Aa_x-Ba_y+2Cb_x+A_x a+A_y b,\\
\label{segunda} 0 & = & b_{xx}+2Ba_x+(2D-A)b_x-Bb_y+B_x a+B_y b,\\
\label{tercera} 0 & = & a_{xy}+(A-D)a_y+Eb_x+Cb_y+C_xA+C_y b,\\
\label{cuarta} 0 & = & b_{xy}+Da_x+Ba_y+(F-C)b_x+D_x a+D_y b,\\
\label{quinta} 0 & = & a_{yy}-Ea_x+(2C-F)a_y+2Eb_y+E_x a+E_y b,\\
\label{sexta} 0 & = & b_{yy}+2Da_y-Eb_x+Fb_y+F_x a+F_y b.
\end{eqnarray}
These correspond, in formula~(\ref{equationkilling}), to the couples $(Y,Z)\in\left\{\left(\indel{x},\indel{x}\right),\left(\indel{x},\indel{y}\right),\left(\indel{y},\indel{y}\right)\right\}$.

The following Lemma is classical in the field (see, for instance, \cite[Thm.~1.6.20]{Wolf} or \cite[\S 3.1]{Szaro}). 

\begin{lemma}[Linearization in exponential coordinates]\label{exponential}  In exponential coordinates around a given  point $m$ in $M$, every Killing vector field of~$\nabla$ vanishing at~$m$ is linear.
\end{lemma}

The main idea in the proof of this  linearization result is that geodesics are integral curves of a vector field in  $TM$  with  quadratic  homogeneous vertical part. In exponential  coordinates at $m$, geodesics passing through $m$ locally rectify on lines with constant speed in $T_mM$  passing through $0$ and any local isometry preserving
$0$ will be linear. In exponential coordinates the Christoffel symbols satisfy  $\Gamma_{ij}^k(0)=0$.  

Lemma~\ref{exponential} implies that a local Killing field of  a connection is completely determined by its $1$-jet at a given point. Consequently, in the neighborhood  of  any point of the manifold, the Lie  algebra of Killing fields  of  a connection is finite-dimensional. This implies that connections are (particular) examples of \emph{rigid geometric structures} in Gromov's sense~\cite{Gro,DG}.

In the analytic setting,  we  will make use of an extendability  result  for local  Killing fields proved first by Nomizu in the Riemannian setting~\cite{Nomizu} and generalized later  for rigid geometric structures  (in particular, for analytic connections) by Amores and Gromov~\cite{Amores,Gro,DG}. This phenomenon states roughly that a  local Killing field of a \emph{rigid analytic} geometric structure can be extended along any curve in the underlying manifold $M$. We then get a multivalued Killing field defined on all of $M$ or, equivalently, a global Killing field defined on the universal cover.

In particular, for analytic connections,  the Killing algebra $\mathfrak{K}(\nabla)$, defined as \emph{the Lie algebra of all local Killing fields   in the neighborhood of a given  point},  does not depend upon the point  (as long as $M$ is connected). We will then call, without ambiguity,   $\mathfrak{K}(\nabla)$, the \emph{Killing algebra of $\nabla$.}

\begin{definition} The connection  $\nabla$ is said to be \emph{locally homogeneous} (and  $\mathfrak{K}(\nabla)$ is said to be  \emph{transitive}) on the open subset $U \subset M$, if for any $u \in U$ and for any tangent vector $V  \in T_{u}U$, there exists a local Killing field $X$ of
$\nabla$ such that $X(u)=V$.
\end{definition}

Nomizu's extension phenomenon doesn't imply that the extension of a family of  pointwise linearly independent Killing fields  stays linearly independent. The rank of the Killing Lie algebra is not constant.  In this setting, we may say that Theorem~\ref{thm-local} classifies, in the case of surfaces,  all the ways in which an analytic symmetric affine connection may degenerate and cease to be locally homogeneous. The following Lemma reduces Theorem~\ref{thm-local} to the study of the case where there is a  Killing algebra of dimension~two:

\begin{lemma}  \label{dim2} Let~$\nabla$ be a real  analytic  affine connection in a neighborhood of the origin in~$\mathbf{R}^2$. Suppose that the Killing algebra~$\mathfrak{K}(\nabla)$ has rank two in an open subset that does not contain the origin, but that contains it in its closure. Then $\mathfrak{K}(\nabla)$ contains a two-dimensional subalgebra which has rank two in some open subset  that does not contain the origin, but that contains it   in its closure.
 \end{lemma}

\begin{proof} Let~$\mathfrak{K}_0(\nabla)\subset\mathfrak{K}(\nabla)$ be the subalgebra consisting of those vector fields that vanish at~$0$ (it is either the full algebra or has codimension~1). By Lemma~\ref{exponential},  elements of $\mathfrak{K}_0(\nabla)$ are linear in exponential coordinates at~$0$. Consequently, $\mathfrak{K}_0(\nabla)$ is embedded in the linear algebra $\mathfrak{gl}(2,\mathbf{R})$. Let~$d$ denote the  dimension of $\mathfrak{K}_0(\nabla)$. 

If $d=1$,  $\mathfrak{K}(\nabla)$ is the sought algebra. 
 
If~$d=4$, $\mathfrak{K}_0(\nabla)= \mathfrak{gl}(2,\mathbf{R})$  and it contains, for example,  the two-dimensional diagonal subalgebra (which is of rank two on an open set containing the origin in its closure).

If~$d=3$,  consider the restriction to~$\mathfrak{K}_0(\nabla)$ of the trace map~$T:\mathfrak{gl}(2,\mathbf{R})\to\mathbf{R}$. If the trace 
map is trivial then~$\mathfrak{K}_0(\nabla)=\mathfrak{sl}(2,\mathbf{R})$, and~$\mathfrak{K}_0(\nabla)$ contains  the two-dimensional parabolic subalgebra which is transitive on an open set accumulating the origin. If the trace map is not trivial, $\mathrm{dim}(\ker(T))=2$. Here $\ker(T)$ is a two dimensional subalgebra of
 $\mathfrak{sl}(2,\mathbf{R})$:  the rank  is two in an open  set accumulating  the origin.

Suppose that~$d=2$, but that~$\mathfrak{K}_0(\nabla)\neq\mathfrak{K}(\nabla)$ and that~$\mathfrak{K}_0(\nabla)$ is of rank one in the neighborhood of the origin, for otherwise there is nothing to prove. Up to a linear change of coordinates, $\mathfrak{K}_0(\nabla)$ is generated by~$X=x\indel{x}$ and~$Y=y\indel{x}$ (the evaluation of any Killing vector field at the origin is tangent to the foliation~$\mathcal{F}$ generated by~$dy$). We have~$[X,Y]=Y$. Let~$Z\in \mathfrak{K}(\nabla)\setminus \mathfrak{K}_0(\nabla)$ be an element  transverse to the foliation~$\mathcal{F}$ in an open set containing the origin in its closure. Suppose that~$[Z,X]=aX+bY+cZ$ and that~$[Z,Y]=dX+eY+fZ$, for some $a,b,c,d,e,f \in \mathbf{R}$. The Jacobi relation applied to~$X$, $Y$ and~$Z$ reads
$(d+af-cd)X+(fb-ce-a)Y+fZ=0$. We must conclude that~$f=0$ and that~$a=-ce$. The equation reduces to
$d(1-c)=0$. If~$d=0$, then~$[Z,Y]=dY$ and the Lie algebra generated by~$Z$ and~$Y$ has the required properties. If~$c=1$, then~$W=-2eX+bY+2Z$ satisfies~$[X,W]=-W$ and the Lie algebra generated by~$X$ and~$W$ 
has the required properties. \end{proof}

\subsection{Locally homogeneous connections and~$(G,X)$-structures}

We finish this section by stating some general results, well known by specialists, showing, in particular,  that a locally homogeneous affine connection on a manifold gives rise to a $(G,X)$-geometry in Thurston's sense~\cite{thurston}. The results will be stated and proved for locally homogeneous affine connections but hold, more generally, for locally homogeneous rigid geometric structures in Gromov's sense~\cite{DG, Gro}.\\

Let~$\nabla$ be an affine connection in a neighborhood of~$0$ in~$\mathbf{R}^n$ and suppose that~$\mathfrak{K}(\nabla)$ is transitive. Let~$G$ be the simply connected Lie group with Lie algebra~$\mathfrak{K}(\nabla)$ and~$G_0\subset G$ the subgroup corresponding to the subalgebra of the Killing vector fields vanishing  at~$0$. \emph{If $G_0$ is closed}, there is a natural identification of a neighborhood of~$G_0/G_0$ in~$G/G_0$ with a neighborhood of the origin in~$\mathbf{R}^n$. When pulling back~$\nabla$ to~$G/G_0$, via this identification, we find a left invariant connection~$\nabla_0$ in~$G/G_0$ which is locally isomorphic  to~$\nabla$.


In our setting, Lemma~\ref{dim2} implies that  a connection satisfying  the hypothesis of the Lemma is,  on a non-trivial open set, locally isomorphic, in the previous sense,   to a connection which is translation-invariant on a two-dimensional Lie group. Since a two-dimensional Lie algebra is  either Abelian or isomorphic to the Lie algebra of the affine group of the real line, we will have to deal with translation-invariant connections either  on $\mathbf{R}^2$,  or  on the affine group. These will be studied, respectively, in sections~\ref{commuKillsub} and~\ref{affine group}.\\

The following Proposition gives a proof of Nomizu's extension of local automorphisms  in the  homogeneous context.

\begin{proposition} Let~$G$ be a connected and simply connected  Lie group $G$  acting transitively on a simply connected manifold $M$ and preserving an affine connection~$\nabla$. Suppose that, at each point~$p\in M$, the Lie algebra of Killing vector fields of~$\nabla$ is the Lie algebra of fundamental  vector fields of  the $G$-action. Let~$U\subset M$ be a connected open set and~$f:U\to M$ a local automorphism  of~$\nabla$. There exists a global automorphism~$\widehat{f}:M\to M$ such that~$\widehat{f}|_U\equiv f$. Moreover, $G$ has finite index in the group of global automorphisms of~$\nabla$.
\end{proposition}
\begin{proof} Let~$\Phi:G\times M\to M$ be the transitive action of~$G$ on~$M$. Let~$\mathfrak{g}$ be the Lie algebra of right-invariant vector fields in~$G$ and let~$\Phi_*:\mathfrak{g}\to\mathfrak{X}(M)$ the induced Lie algebra morphism. Since~$f$ is a automorphism, it maps Killing vector fields of~$\nabla$ into Killing vector fields of~$\nabla$ and thus induces an automorphism~$f_*:\mathfrak{g}\to\mathfrak{g}$. In its turn, this induces an automorphism~$f_*^\sharp:G\to G$. Let 
$$F (g,p)=\Phi(f_*^\sharp(g),f(\Phi(g^{-1},p))),$$
wherever it is defined (in a subset of~$G\times M$ and taking values in~$M$). For~$g$ fixed, $F(g,p)$ is a automorphism. Notice that,  for~$p \in U$, $F(e,p)=f(p)$. We will show that~$F(g,p)$ is in fact independent of~$g$. Since for every~$p\in M$ there exists~$g\in G$ such that~$(g,p)$ is in the domain of definition of~$F$, we may define~$\widetilde{f}(p)$ as~$F(g,p)$.  For~$g_0\in G$ and~$v\in T_{g_0}G$, identified to the element of~$\mathfrak{g}$ whose value at~$g_0$ is~$v$, we have
\begin{eqnarray*}\left.\frac{\partial F}{\partial g}\right|_{g_0}(v) &  = & \Phi_*(f_*v)+\frac{\partial \Phi}{\partial p}\circ Df\circ  \Phi_*(\mathrm{Ad}_{{g_0}^{-1}}(-v)) \\
&  = & \Phi_*(f_*v)-\frac{\partial \Phi}{\partial p}\circ  \Phi_*(f_*\mathrm{Ad}_{g_0^{-1}}(v)) \\
&  = & \Phi_*(f_*v)-\Phi_*(\mathrm{Ad}_{f_*^\sharp (g_0)}(f_*\mathrm{Ad}_{g_0}^{-1}(v))) \\
&  = & \Phi_*(f_*v)-\Phi_*(f_*\mathrm{Ad}_{g_0} f_*^{-1}(f_*\mathrm{Ad}_{g_0}^{-1}(v))) =0. 
\end{eqnarray*}
and~$F$ is thus independent of~$g$. Hence, a local automorphism is in fact the restriction of a global one. Hence, the automorphism  pseudogroup of~$\nabla$ is the pseudogroup associated to a group action. If~$f:M\to M$ is a global isometry, there is an element~$g\in G$ such that~$g\circ f$ fixes a point~$p\in M$. But the group of germs of automorphisms  fixing a point contains the group generated by the elements of the Killing algebra that vanish at the point with finite index (see~\cite{DG}  3.5, or~\cite{Gro}  3.4.A).  This implies that~$[\mathrm{Isom}(\nabla):G]<\infty$. 
\end{proof}

Finally, the following proposition establishes the link between~$(G,X)$-structures and locally homogeneous connections.
\begin{proposition}\label{constgx} Let~$M$ be a manifold endowed with a locally homogeneous  affine connection $\nabla$. Let~$p\in M$, let~$\mathfrak{g}$ be the Lie algebra of Killing vector fields of~$\nabla$ in a neighborhood of~$p$ and let~$\mathfrak{g}_0\subset \mathfrak{g}$ the subalgebra of those Killing vector fields vanishing at~$p$. Let~$G$ be the connected and simply connected Lie group corresponding to~$\mathfrak{g}$ and~$G_0\subset G$ the subgroup corresponding to~$\mathfrak{g}_0$. If~$G_0$ is closed, then there exists a finite covering~$\overline{\pi}:\overline{M}\to M$, a $G$-invariant connection~$\nabla_0$ on~$G/G_0$ and a~$(G,G/G_0)$-\nobreakdash structure on~$\overline{M}$, such that~$\overline{\pi}^*\nabla$ is locally isomorphic to $\nabla_0$. 
\end{proposition}

\begin{proof} 
Let~$X=G/G_0$ and  $x_0=G_0/G_0$. There is a neighborhood~$V\subset M$ and a diffeormorphism~$\psi:(V,p)\to (X,x_0)$ mapping the Killing vector fields of~$\nabla$ to the vector fields induced by~$G$ on~$X$. Let~$\nabla_0^V$ be the connection in~$\psi(V)$ given by~$\psi(\nabla)$. By construction, the vector fields in~$X$ inducing the action of~$G$ are Killing vector fields of~$\nabla_0^V$. Propagate~$\nabla_0^V$ to all of~$X$ via the action of~$G$. This produces a globally-defined $G$\nobreakdash-invariant affine connection~$\nabla_0$ in~$X$, which is locally isomorphic to $\nabla$.

Let~$q\in M$. Because the Killing algebra of~$\nabla$ has full rank, by composing finitely many local automorphisms of those generated by integrating the Killing algebra, we may find a  local  automorphism~$f:(U,q)\to(M,p)$. Hence, for~$h=\psi\circ f$, $h:U\to X$ maps~$\nabla$ to~$\nabla_0$. We thus have an atlas~$\{(U_i,h_i); i\in I\}$ of~$M$ taking values in~$X$ whose changes of coordinates lie within~$\mathrm{Isom}(\nabla_0)$. We will now proceed in a manner akin to the construction of the orientable double cover of a manifold. Let~$\mathrm{Isom}(\nabla_0)/G=\{[k_1],\ldots, [k_n]\}$. For $j \in \{1, \ldots, n \}$, let~$V_i^j=U_i$. The manifold~$W=\sqcup V_i^j$ is naturally endowed with a canonical projection~$\pi:W\to M$ and charts~$k_j\circ f_i:V_i^j \to X$. Within~$W$, identify~$x\in V_i^j$ with~$x'\in V_{i'}^{j'}$, if~$V_i^j\cap V_{i'}^{j'}\neq \emptyset$ and if~$(k_j\circ h_i)\circ (k_{j'}\circ h_{i'})^{-1}$ is the restriction of an element of~$G$. After the identification, we have a manifold~$\overline{M}$ and an~$n$\nobreakdash-fold covering~$\overline{\pi}:\overline{M}\to M$  that is canonically endowed with an atlas taking values in~$X$ whose changes of coordinates are in~$G$. (If~$\overline{M} $ is not connected we may replace it by one of its connected components). \end{proof}

\begin{remark}\label{rem-gxorient} In the past proof, if~$M$ is oriented and~$X$ is orientable, $X$ may be oriented by~$\psi$. In this case, by construction, the charts~$h_i:U_i\to X$ preserve the orientation and we may replace, in the above construction, $\mathrm{Isom}(\nabla_0)$, by its subgroup of orientation-preserving automorphisms~$\mathrm{Isom}^+(\nabla_0)$. In particular, if~$G=\mathrm{Isom}^+(\nabla_0)$, the finite covering of Proposition~\ref{constgx} may be chosen to be trivial.
\end{remark}

\begin{remark} The~$(G,X)$\nobreakdash-structure associated to~$M$ in Proposition~\ref{constgx} is not canonically associated to~$\nabla$. The subgroup~$S_{G_0}$ of automorphisms of~$G$ that preserve~$G_0$ acts naturally upon the $G$\nobreakdash-invariant connections on~$G/G_0$ and the connection~$\nabla_0$ of Proposition~\ref{constgx} may be replaced by any one 
along the same orbit of~$S_{G_0}$, producing a different~$(G,X)$\nobreakdash-structure. \end{remark}

\section{Quasihomogeneous connections, commutative Killing subalgebra}   \label{commuKillsub}

The aim of this section is to prove Theorem~\ref{thm-local} in the case where the Killing algebra of dimension two constructed in Lemma~\ref{dim2}  is commutative. We will prove:

\begin{proposition} \label{propcom} 
Let~$\nabla$ be a real-analytic torsion-free affine connection  defined in a neighborhood of the origin in~$\mathbf{R}^2$ and suppose that~$\nabla$ admits Killing vector fields~$X$ and~$Y$ which commute and that are linearly independent in an open set having the origin in its closure. Then~$\nabla$ is locally homogeneous in the neighborhood of the origin.
\end{proposition}

Let~$\nabla$ be a connection defined in a neighborhood of the origin of~$\mathbf{R}^2$ with Christoffel symbols~$\{\Gamma_{ij}^k\}$ and let~$X$ and~$Y$ be commuting Killing fields as in the statement of Proposition~\ref{propcom}. By Lemma~\ref{killingcentralizer} and the fact that~$X$ and~$Y$ generate the Lie algebra of vector fields that commute with~$X$ and~$Y$,
\begin{equation}\label{consconn}\nabla_{X}X=\alpha X+\beta Y,\;\nabla_{X}Y=\gamma X+\delta Y,\;\nabla_{Y}Y=\epsilon X+\phi Y.\end{equation}
for some constants~$\alpha, \beta, \gamma, \delta, \epsilon, \phi$.

\subsection{One-dimensional orbits}   We will suppose, without loss of generality, that~$X(0)=0$ and that it is, following Lemma~\ref{exponential}, linear. Since~$X$ and~$Y$ commute, $X$ must be a constant multiple of~$Y$ along its common orbit and, since~$X$ vanishes at the origin, it must have one vanishing eigenvalue (and thus both of its eigenvalues are real). There are two cases to be considered, according to the diagonalizability of~$X$.\\

\paragraph{Diagonalizable} We may assume that
$X=x\indel{x}$. It vanishes along~$x=0$ and thus this curve is invariant by~$Y$. If~$Y=xf\indel{x}+g\indel{y}$ commutes with~$X$, then~$f_x=0$ and~$g_x=0$. Hence, $f=f(y)$ and $g=g(y)$ with~$g(0)\neq 0$. We may suppose, up to a change of variables in~$y$, that~$g\equiv 1$. In the coordinates~$(e^{-\int_0^y h(s)ds}x,y)$ the vector fields are
$$X =x\del{x},\; Y=\del{y}.$$
On the other hand, from the expressions of $X$ and $Y$ in the coordinates~$(x,y)$,  we can compute $\nabla_XX, \nabla_XY$ and $\nabla_YY$ in terms of the Christoffel symbols $\Gamma_{ij}^k$. Solving for the Christoffel symbols, we have
$$\Gamma_{11}^1= \frac{\alpha-1}{x}, \;\Gamma_{11}^2= \frac{\beta}{x^2}, \; \Gamma_{12}^1=\gamma ,\;  \Gamma_{12}^2= \frac{\delta}{x},\; \Gamma_{22}^1= \epsilon x, \;\Gamma_{22}^2=\phi .$$
Since the connection is analytic at the origin, $\alpha=1$, $\beta=0$, $\delta=0$ and the connection is
\begin{equation}\label{connspeciale}
\nabla_{\del{x}}\del{x}=0,\; \nabla_{\del{x}}\del{y}=\gamma\del{x},\; \nabla_{\del{y}}\del{y}=\epsilon x\del{x}+\phi\del{y}.
\end{equation}
We will prove that such a connection is locally homogeneous at the origin. A vector field~$h(y)\indel{x}$ is a Killing vector field of~(\ref{connspeciale}) if and only if it satisfies the system of equations~(\ref{primera}--\ref{sexta}), which reduce to the single ordinary differential equation $ h''+(2\gamma-\phi)h'+\epsilon h=0$. 
If~$h_0(y)$ is a solution to this equation that does  not vanish at the origin, the Killing vector fields~$\indel{y}$ and~$h_0(y)\indel{x}$ are linearly independent at the origin. \\

\paragraph{Non-diagonalizable case} We have
$X=x\indel{y}$. If~$Y=f\indel{x}+g\indel{y}$, with $f$ and $g$ analytic functions,  then, from~$[X,Y]=0$, 
$f_y=0$ and $x g_y-f=0$. Hence, there exist analytic  functions~$h(x)$, $k(x)$ such that~$f=xh, g=yh+k$ with~$k(0)\neq 0$.  This leads to the following expression:
$$X=x\del{y},\; Y=xh(x)\del{x}+(yh(x)+k(x))\del{y}.$$
As for the equations defining the Christoffel symbols  in the coordinates $(x,y)$, we have
$$\nabla_{X}X=x^2\left(\Gamma_{22}^1\del{x}+\Gamma_{22}^2\del{y}\right)=\alpha x\del{y}+\beta\left[xh\del{x}+(yh+k)\del{y}\right].$$
Evaluating at~$0$, the equation for~$\indel{y}$ reads~$\beta k(0)=0$ and thus~$\beta=0$. This implies that~$\alpha=0$,~$\Gamma_{22}^1\equiv 0$ and~$\Gamma_{22}^1\equiv 0$. We also have
$$\nabla_{X}Y=x^2h\left(\Gamma_{12}^1\del{x}+\Gamma_{12}^2\del{y}\right)+xh\del{y}=\gamma x\del{y}+\delta\left[xh\del{x}+(yh+k)\del{y}\right].$$
Again, evaluating at~$0$ in the equation for~$\indel{y}$ gives~$\delta=0$ and thus~$\Gamma_{12}^1\equiv 0$. We obtain
\begin{equation}\label{eqbanale}\Gamma_{12}^2=(\gamma h^{-1}-1)x^{-1}.\end{equation} Since $\Gamma_{12}^2$ must be analytic at the origin, $h(0)\neq 0$. Up to multiplying~$Y$ by a constant we will suppose that~$h(0)=1$.  

We show now that we can find new local coordinates in the neighborhood of the origin, with respect to which
$X$ remains as before, but where we can suppose that~$h\equiv 1$ in the expression for~$Y$. For this, let $\xi$ be an analytic function in $x$, with~$\xi(0)\neq 0$ and such that, in the coordinate $\overline{x}=x\xi(x)$ the vector field~$xh(x)\indel{x}$ reads~$\overline{x}\indel{\overline{x}}$. Then,  in the coordinates $(\overline{x},\overline{y})= \left(x\xi(x),y\xi(x)\right)$, we have
$$X=\overline{x}\del{\overline{y}}, \; Y=\overline{x}\del{\overline{x}}+(\overline{y}+k(\overline{x}))\del{\overline{y}}.$$
Let us keep the notation $(x,y)$ for these new coordinates. We thus have, from~(\ref{eqbanale}), $\gamma=1$ and~$\Gamma_{12}^2\equiv 0$.  For the last coefficients of the connection we have
$$\nabla_{Y}Y-Y=x^2\left(\Gamma_{11}^1\del{x}+\Gamma_{11}^2\del{y}\right)+xk'\del{y}=\epsilon x\del{y}+(\phi-1)\left[x\del{x}+(y+k)\del{y}\right].$$
Evaluating at the origin, we must have~$\phi=1$ and hence~$\Gamma_{11}^1\equiv 0$. All the Christoffel symbols vanish except for~$\Gamma_{11}^2=(\epsilon-k')x^{-1}$. Since~$\partial \Gamma_{11}^2/\partial y=0$,  the curvature of $\nabla$ vanishes and $\nabla$  is flat and, in particular, locally homogeneous at the origin.

\subsection{Fixed points} We now assume that the Lie algebra generated by $X$ and $Y$ is of rank $0$ at the origin. By Lemma~\ref{exponential}, there are coordinates~$(x,y)$ where both vector fields are linear. If~$Z$ and~$W$ are linear vector fields,
$$\nabla_Z W=L+\sum Q_{ij}^k \Gamma_{ij}^k,$$
for some linear vector field~$L$ and some quadratic and homogeneous functions~$Q_{ij}^k$. In this way, equations (\ref{consconn}) become a system of the form
$$\sum_{i,j,k}  Q_{ij}^{kl} \Gamma_{ij}^k=L_l$$
for~$l=1,\ldots, 3$, corresponding to the pairs~$(Z,W)\in\{(X,Y), (X,X), (Y,Y)\}$.  But since~$\Gamma_{ij}^k$ is analytic at the origin and the left-hand side of each of these equations has a trivial linear part at the origin, we must have~$L_l=0$.  The system has a unique solution since the connection is determined by~(\ref{consconn}) and thus~$\Gamma_{ij}^k\equiv0$. The connection is flat and thus locally homogeneous. This finishes the proof of  Proposition~\ref{propcom}.\\

Remark that the above argument works also in higher dimensions. It shows that if~$\nabla$ is an analytic  connection in a neighborhood of the origin of~$\mathbf{R}^n$ and~$X_1$, \ldots, $X_n$ are commuting Killing vector fields of~$\nabla$ vanishing at the origin, but of rank $n$ on some nontrivial open set, then  $\nabla$  is flat.

\section{Invariant connections on the affine group} \label{affine group}
The \emph{group of affine transformations of the real line} or, simply, the \emph{affine group} is the group
$$\mathrm{Aff}(\mathbf{R})=\{x\mapsto ux+v,\; u,v\in\mathbf{R}, u\neq 0 \},$$
with the product
$$(u_1,v_1)\cdot(u_2,v_2)=(u_1 u_2,u_1v_2+v_1)$$
and Lie algebra~$\mathfrak{aff}(\mathbf{R})$. We will denote by~$\mathrm{Aff}_0(\mathbf{R})$ the connected component of the identity (the couples~$(u,v)$ with~$u>0$). In $\mathrm{Aff}_0(\mathbf{R})$, the Lie algebra of~\emph{right-invariant} vector fields is generated by
\begin{equation}\label{rightinv} A_0=-\left(u\del{u}+v\del{v}\right),\;B_0= \del{v},\end{equation}
which satisfy~$[A_0,B_0]=B_0$. The Lie algebra of \emph{left-invariant} vector fields is generated by
$$X_0=u\del{u},\; Y_0=-u\del{v},$$
which satisfy~$[X_0,Y_0]=Y_0$. The Lie algebra of left-invariant vector fields may be locally characterized as that of the vector fields which simultaneously commute with~$A_0$ and~$B_0$. A left-invariant vector field will be called~\emph{semisimple} if it is conjugate to~$X_0$ and \emph{unipotent} if it is conjugate to~$Y_0$. The group of automorphisms of~$\mathrm{Aff}_0(\mathbf{R})$ is the restriction to~$\mathrm{Aff}_0(\mathbf{R})$ of the group of inner automorphisms of~$\mathrm{Aff}(\mathbf{R})$. Via the action by inner automorphisms (adjoint action) we have
\begin{equation}\label{inner} (u,v)_*\left(\begin{array}{c} B_0 \\ A_0\end{array}\right)=\left(\begin{array}{cc} u & 0 \\ v & 1\\ \end{array}\right)\left(\begin{array}{c} B_0 \\ A_0\end{array}\right), \;\; (u,v)_*\left(\begin{array}{c} Y_0 \\ X_0\end{array}\right)=\left(\begin{array}{cc} u & 0 \\ v & 1\\ \end{array}\right)\left(\begin{array}{c} Y_0 \\ X_0\end{array}\right).
\end{equation}

In particular, the orientation-preserving automorphisms of~$\mathrm{Aff}_0(\mathbf{R})$ are inner ones.\\

If a connection~$\nabla$ in~$\mathrm{Aff}_0(\mathbf{R})$ is invariant by left-translations, its  Killing algebra $\mathfrak{K}(\nabla)$ contains the Lie algebra generated by $A_0$ and $B_0$. Since~$A_0$ and~$B_0$ commute with~$X_0$ and~$Y_0$, $\nabla$ must be constant with respect to these left-invariant vector fields in the sense that there must exist constants $\alpha$, $\beta$, $\gamma$, $\delta$, $\epsilon$ and~$\phi$ such that
\begin{equation}\label{affinvcon}\nabla_{X_0}X_0=\alpha X_0+\beta Y_0,\;\nabla_{X_0}Y_0=\gamma X_0+\delta Y_0,\;\nabla_{Y_0}Y_0=\epsilon X_0+\phi Y_0.\end{equation}

The group of automorphisms of~$\mathrm{Aff}_0(\mathbf{R})$ acts upon the left-invariant connections (see formula~(\ref{inner}) for the action on left-invariant vector fields): when replacing~$(X_0,Y_0)$ by~$(X,Y)=(X_0+\lambda Y_0,\mu Y_0)$,
\begin{eqnarray}\nonumber \nabla_{X}X & = & (\alpha+2\lambda\gamma+\epsilon\lambda^2)X+\mu^{-1}(\beta+[2\delta-\alpha-1]\lambda+[\phi-2\gamma]\lambda^2 -\lambda^3\epsilon)Y  ,\\ \label{actionadjconn}
\nabla_{X } Y &= &\mu(\gamma+\lambda\epsilon) X +(\delta+[\phi-\gamma]\lambda-\epsilon\lambda^2) Y ,\\ \nonumber \nabla_{Y }Y  & = & \mu^2\epsilon X +\mu(\phi-\lambda\epsilon) Y .
\end{eqnarray}

\begin{definition}\label{defnmarking} A \emph{marked} connection on~$\mathrm{Aff}_0(\mathbf{R})$ is a couple~$(\nabla,Z)$ of a left-invariant connection~$\nabla$ and a left-invariant semisimple vector field~$Z$ such that~$\nabla_Z Z\wedge Z=0$. We say that~$Z$ is a \emph{marking} of~$\nabla$. In a marked connection~$(\nabla,Z)$, $\alpha(\nabla,Z)$ is the unique real such that~$\nabla_Z Z=\alpha(\nabla,Z)Z$ and~$\delta(\nabla,Z)$ is the unique real such that~$\nabla_Z Y_0=\gamma Z+\delta(\nabla,Z) Y_0$. A marked connection~$(\nabla,Z)$ is said to be
\begin{itemize}\item of Type~$\mathrm{I}_0(n)$, if~$\alpha(\nabla,Z)=-1/n$ and~$\delta(\nabla,Z)=1$;
 \item of Type~$\mathrm{II}_0(n)$, if~$\alpha(\nabla,Z)=-1/n$ and~$\delta(\nabla,Z)=1+\alpha(\nabla,Z)$,
\end{itemize}
for some $n \in \mathbf{R}^*$.

A marked connection is said to be \emph{special}, if it is  either of  Type~$\mathrm{I}_0(n)$, or of Type~$\mathrm{II}_0(n)$, with~$n\in\frac{1}{2}\mathbf{Z}$, $n\geq 1$.
\end{definition}

The interest of this Definition comes from the fact that quasihomogeneous  analytic connections on compact surfaces will be locally modeled, in the locus of local homogeneity, by special connections of Type~$\mathrm{I}_0(n)$, or of Type~$\mathrm{II}_0(n)$, with respect to some marking. Formula~(\ref{actionadjconn}) guarantees that, in a marked connection~$(\nabla,Z)$, the quantities~$\alpha(\nabla,Z)$ and~$\delta(\nabla,Z)$ are well-defined. From the same formula, it follows that a connection has, if finitely many, at most three markings.

\begin{remark}\label{deuxtypes}
Let~$\nabla$ be a connection of the form (\ref{affinvcon}) with~$\beta=0$, this is, $(\nabla,X)$ is a marked connection. We have~$\alpha(\nabla,X)=\alpha$ and~$\delta(\nabla,X)=\delta$. For~$\lambda\neq 0$, $(\nabla,X+\lambda Y)$ is a marked connection with $\alpha(\nabla, X+\lambda Y)=\alpha'$ and~$\delta(\nabla, X+\lambda Y)=\delta'$ if and only if
\begin{equation}\label{iftwomarkings}\left(\lambda \gamma,\lambda^2\epsilon,\lambda \phi\right)=(\delta+\delta'-1-\alpha,\alpha+\alpha'-2\delta-2\delta'+2,1-2\delta+\alpha').\end{equation}
In particular, the connection is determined by the values of~$\alpha$, $\alpha'$, $\delta$ and~$\delta'$  up to the natural equivalence~(\ref{actionadjconn}). For the third marking~$X+\lambda''Y$ we have
\begin{equation}\label{nothreespmarkings}\alpha''=\frac{\alpha\alpha'-1-4\delta\delta'+2\delta+2\delta'}{2-2\delta-2\delta'+\alpha+\alpha'},\; \delta''=\frac{1-\delta\alpha'-\alpha\delta'+\alpha\alpha'+\alpha+\alpha'-\delta-\delta'}{2-2\delta-2\delta'+\alpha+\alpha'}.\end{equation}
\end{remark}

A left-invariant connection on the affine group might admit some extra symmetries: the full isometry group might be bigger than the one generated by Killing algebra. In the cases that will interest us, we have:

\begin{proposition}\label{isomleft} Let~$\nabla_0$ be a left-invariant connection in~$\mathrm{Aff}_0(\mathbf{R})$ such that~$\mathfrak{K}(\nabla_0)$ is the Lie algebra of right-invariant vector fields. 
\begin{enumerate}
\item\label{isomleftuno} If~$\nabla_0'$ is a left-invariant connection in~$\mathrm{Aff}_0(\mathbf{R})$, then any germ of isometry between~$\nabla_0$ and~$\nabla_0'$, fixing the identity, is given by an inner automorphism of~$\mathrm{Aff}(\mathbf{R})$.
\item\label{isomleft-ori} Any germ of orientation-preserving self-isometry of~$\nabla_0$ is the germ of a left translation (equivalently, a germ of orientation-preserving isometry of~$\nabla_0$ fixing a point is the identity).

\end{enumerate}
\end{proposition}
\begin{proof} (\ref{isomleftuno}) Since the two connections are isometric, the dimension of the killing algebra is the same and thus~$\mathfrak{K}(\nabla_0')$ is  is the Lie algebra of right-invariant vector fields.  The isometry must map~$\mathfrak{K}(\nabla_0)$ into~$\mathfrak{K}(\nabla_0')$, this is, must preserve the Lie algebra of right-invariant vector fields at~$e$. Every automorphism of~$\mathfrak{aff}(\mathbf{R})$ comes from the adjoint action of~$\mathrm{Aff}(\mathbf{R})$. Hence the germ of~$\phi$ at the identity is the germ of an inner automorphism of~$\mathrm{Aff}(\mathbf{R})$. (\ref{isomleft-ori}) Consider now the case where~$\nabla_0'=\nabla_0$ and suppose that the inner automorphism is the one associated to~$(X_0,Y_0)\mapsto (X_0+\lambda Y_0,\mu Y_0)$. The associated constants are related according to~(\ref{actionadjconn}): if a connection of the form~(\ref{affinvcon}) is preserved by such an automorphism then the connection~(\ref{affinvcon}) must equal the connection~(\ref{actionadjconn}) when the frames~$(X_0,Y_0)$ and~$(X,Y)$ are identified. The corresponding equations are easy to solve for they are linear in~$\alpha$, $\beta$, $\gamma$, $\delta$, $\epsilon$ and~$\phi$.

If~$\mu\neq -1$, a straightforward computation shows that we must have~$(\gamma,\epsilon,\phi)=(0,0,0)$. In this case, $v\indel{v}$ is a Killing field (which is not right-invariant). If~$\mu=-1$, then any potential isometry has order two and is orientation-reversing.
\end{proof}

\subsection{Killing algebras of connections of Types~$\mathrm{I}$ and~$\mathrm{II}$} The Killing Lie algebra of a left-invariant connection may be  strictly bigger that the Lie algebra of right-invariant vector fields. We will now deal with this question and compute, for the connections of special type of Definition~\ref{defnmarking}, the full Killing algebra.

In the coordinates~$(x,y)=(\log(u),v)$ of~$\mathrm{Aff}_0(\mathbf{R})$,  we have~$X =\indel{x}$, $Y =-e^x\indel{y}$. The connection is thus of the form~(\ref{genericcon}) for
\begin{equation}\label{consy}A=\alpha , \; B=-\beta e^x,  \; C= -\gamma e^{-x}, \; D=\delta-1, \;  E=\epsilon e^{-2x},  \;F=-\phi e^{-x}.\end{equation}

Equation~(\ref{equationkilling})  implies that~$a(x,y)\indel{x}+b(x,y)\indel y$ is a Killing vector field of the connection if its coefficients satisfy the system (\ref{primera}--\ref{sexta}) with the above values of the corresponding functions.

A two-dimensional space of solutions of these equations is given by the vector fields coming from~(\ref{rightinv}) since, by construction, our  connections are left-invariant and hence right-invariant vector fields (generating left translations) are Killing fields. This space is generated by~$(a,b)=(1,y)$ and~$(a,b)=(0,1)$. We will find all the solutions to these equations in two particular  cases, related to special connections of Type~$\mathrm{I}_0$ and~$\mathrm{II}_0$ of Definition~\ref{defnmarking}.

\begin{proposition}\label{typei-ii} For~$\alpha=-1/n$, $n\in\frac{1}{2}\mathbf{N}^*$, $\beta=0$ and parameters~$\epsilon$, $\gamma$ and~$\phi$ that do not vanish simultaneously, for the system~(\ref{primera}--\ref{sexta}) with~(\ref{consy}) we have:
\begin{enumerate}
 \item If~$\delta=1$, the space of solutions has dimension~2 except in the case when~$\alpha=-1$, $\epsilon=-\gamma^2$ and $\phi=-\gamma$. 
 \item If~$\delta=1+\alpha$ and~$\alpha\neq -1$, the space of solutions has dimension~2 except in the case when $\alpha=-1/2$ and $\phi=2\gamma$.
\end{enumerate}
In both exceptional cases the space of solutions has dimension~3 and is isomorphic to~$\mathfrak{sl}(2,\mathbf{R})$.
\end{proposition}

\begin{proof} Let us begin by the first case, $\delta=1$. Equation (\ref{segunda}) yields
$b= f_1(y)+f_2(y)e^{\alpha x}$ for some functions~$f_1$ and~$f_2$. Equation~(\ref{cuarta}) becomes
$e^xf_2'=(\phi-\gamma)f_2$. Since~$f_2$ is a function of~$y$, both sides of the equation must vanish. Now we consider two cases: $\phi \neq \gamma$ and $\phi=\gamma$.
\begin{enumerate}

\item  if $\phi\neq \gamma$, then  $f_2\equiv 0$. From~(\ref{primera}), $a=f_3(y)+f_4(y)e^{-\alpha x}$. Equation~(\ref{tercera}) minus equation~(\ref{sexta}) becomes
$$[\alpha f_3'-f_1'']e^{x(1+\alpha)}+[(\gamma-\phi)(f_3-f_1')]e^{\alpha x}+[(\gamma-\phi) f_4]=0.$$
The expressions in brackets are functions of~$y$ and thus, if there are no linear relations between~$1$, $e^{x(1+\alpha)}$ and~$e^{\alpha x}$ (this is, if~$\alpha\neq -1$),  the expressions in brackets must vanish. In particular, we must have~$f_3=f_1'$.

If~$\alpha\neq -1$, $f_4\equiv 0$ and the equation becomes $f_1''\equiv 0$. The space of solutions is two-dimensional. 

If~$\alpha=-1$, $f_4=2f_1''(\gamma-\phi)^{-1}$. Equation~(\ref{sexta}) is now~$(\gamma+\phi)f_1''=0$. If~$\phi+\gamma\neq 0$, then~$f_1''=0$ and $f_4=0$. As before,  the space of solutions is two-dimensional. Otherwise, $\phi=-\gamma$. In this last case we are left with equation~(\ref{quinta}):
$$e^{2x} f_1^{(\text{iv})}-2e^x \gamma f_1'''-3(\epsilon+\gamma^2)f_1''=0.$$
We must then have~$\epsilon=-\gamma^2$, $f_1'''=0$ . The space of solutions has dimension~$3$. An extra solution is given by~$(a,b)=(y+\gamma^{-1}e^x,\frac{1}{2}y^2)$ and the corresponding Lie algebra is isomorphic to~$\mathfrak{sl}(2,\mathbf{R})$.
\item  if $\phi=\gamma$, then $f_2$ is constant. Either
\begin{enumerate}
\item $\gamma=0$. Equation~(\ref{primera}) becomes~$a_{xx}+\alpha a_x=0$ and thus
$a=f_3(y)+f_4(y)e^{-\alpha x}$.
Equation~(\ref{sexta}) reads~$f_1''=\alpha\epsilon e^{(\alpha-2)x}f_2$ and equation~(\ref{tercera}) is $f_3'=-\epsilon  e^{(\alpha-2)x}f_2$. Since~$\epsilon\neq 0$ and~$\alpha(\alpha-2)\neq 0$, $f_2\equiv0$, $f_3'=0$ and~$f_1''=0$ and we are left with equation~(\ref{quinta}):
$$e^{2x}f_4''+\epsilon (\alpha-2)f_4+2\epsilon e^{\alpha x}(f_1'-f_3)=0;$$
we must have $f_4\equiv 0$ and~$f_3=f_1'$. The space of solutions is two-dimensional.
\item $\gamma\neq 0$. We can solve for~$a(x,y)$ in equation~(\ref{sexta}) and obtain
$$a(x,y)=f_1'-\gamma^{-1}e^xf_1''+\alpha\epsilon\gamma^{-1}e^{(\alpha-1)x}f_2.$$
Equation~(\ref{tercera}) is
$$e^{3x}(\alpha+1)f_1'''-\gamma(\alpha-1)e^{2x}f_1''-2\gamma\epsilon\alpha  e^{\alpha x}f_2.$$
and thus~$f_1''=0$. The equation becomes~$\epsilon f_2=0$. If~$\epsilon\neq 0$ then~$f_2=0$ and the space of solutions is two-dimensional. If~$\epsilon=0$ the first equation reduces also to~$f_2=0$. 
\end{enumerate}
\end{enumerate}
This proves the Proposition in the case~$\delta=1$. Let us assume, from now on, that~$\delta=1+\alpha$. From equation~(\ref{segunda}), $b=f_1(y)+f_2(y)e^{-\alpha x}$. Solving for~$a$ in~(\ref{cuarta}) we get
$$a(x,y)=\frac{\phi-\gamma}{(\alpha+1)e^{(\alpha+1)x}}f_2-\frac{1}{\alpha e^{\alpha x}}f_2'+f_3(y).$$
Equation~(\ref{primera}) is now~$(\phi+2\gamma\alpha-\gamma)f_2=0$; equations~(\ref{tercera}), (\ref{quinta}) and~(\ref{sexta}) are
$$f_2''+[c_1f_2'] e^{-x}+\left[\frac{ \gamma\phi-\epsilon\alpha-\epsilon\alpha^2-\gamma^2}{\alpha+1}f_2\right]e^{-2x}+\left[\gamma (f_3-f_1')\right]e^{(\alpha-1)x}=0,$$
$$[c_2f_2''']e^{x} + c_3f_2''+[c_4 f_2']e^{-x}+\left[\frac{\epsilon(\alpha-1)(\phi-\gamma)}{1+\alpha}f_2\right]e^{-2x}+f_3''e^{(\alpha+1)x}+[c_5f_3']e^{\alpha x}-[2\epsilon (f_3-f_1')]e^{(\alpha-1)x}=0,$$
$$f_2''+[c_6f_2']e^{-x}-\left[\frac{\epsilon\alpha-\phi\gamma+\epsilon\alpha^2+\phi^2}{\alpha+1}f_2\right]e^{-2x}-[\phi(f_3-f_1')] e^{(\alpha-1)x}-[2\alpha f_3'+f_1'']e^{\alpha x}=0,$$
for some constants~$c_i$. In this equation, each summand is a function of~$y$ multiplying some power of~$e^x$. The conditions imposed on~$\alpha$ guarantee that, in the last three equations, the power of~$e^x$ that appears as coefficient for~$f_3-f_1'$ is different from the powers appearing in the other summands. Since~$\epsilon$, $\gamma$ and~$\phi$ do not vanish simultaneously, $f_3=f_1'$. By the same argument, the coefficients of~$f_2$ in these four equations do not all vanish and thus $f_2\equiv 0$. The system reduces to the equations
$$e^xf_1'''=(2\gamma-\phi)f_1'',\;(1+2\alpha)f_1''=0.$$

Either $f_1''=0$ and the space of solutions has dimension~$2$, or~$\alpha=-1/2$, $\phi=2\gamma$ and~$f_1'''=0$:  the space of solutions has dimension three and an extra solution is given by~$(a,b)=(2y,y^2)$: the Killing algebra is isomorphic to~$\mathfrak{sl}(2,\mathbf{R})$. 
\end{proof}

\section{Quasihomogeneous connections, affine Killing subalgebra}   \label{Quaconnaffine}
The aim of this section is to prove Theorem~\ref{thm-local} in the case where the Killing algebra of dimension two constructed in Lemma~\ref{dim2}  is  isomorphic to $\mathfrak{aff}(\mathbf{R})$. Following the terminology of Theorem~\ref{thm-local}, we will prove:

\begin{proposition}\label{propaffine}   Let~$\nabla$ be a symmetric analytic connection in a neighborhood of the origin of~$\mathbf{R}^2$ and let~$A$ and~$B$ be Killing fields of~$\nabla$ such that~$[A,B]=B$ and such that~$A$ and~$B$ are linearly independent in an open set that accumulates to the origin,  but that does not contain the origin. 
\begin{itemize}
\item If~$\mathfrak{K}(\nabla)$ is the one generated by~$A$ and~$B$ then: 
\begin{enumerate} 
\item\label{pa1}  If $A$ and $B$ vanish at the origin, then $\nabla$ is locally isomorphic   to a connection of  Type~$\mathrm{II}^0$.

\item\label{pa2} If the Lie algebra generated by $A$ and $B$ is of rank $1$ at the origin and the isotropy algebra at the origin  is generated by a  semisimple element , then $\nabla$ is locally isomorphic  to a connection of  Type~$\mathrm{I}$.

\item\label{pa3} If the Lie algebra generated by $A$ and $B$ is of rank $1$ at the origin and the isotropy algebra at the origin  is generated by a  unipotent  element, then $\nabla$ is locally isomorphic  to a connection of  Type~$\mathrm{II}^1$.
\end{enumerate}
\item If~$\mathfrak{K}(\nabla)$ contains properly the Lie algebra generated by~$A$ and~$B$,  then either~$\nabla$ is flat,  or  $\nabla$ is locally isomorphic to a connection of Type~$\mathrm{III}$. In this last case, $\mathfrak{K}(\nabla)\approx\mathfrak{sl}(2,\mathbf{R})$.
\end{itemize}
\end{proposition}

In subsection~\ref{fixed points} we prove part (\ref{pa1}) of Proposition~\ref{propaffine}. Parts (\ref{pa2}) and (\ref{pa3}) will be respectively proved in sections~\ref{Semi-simplestabilizer} and~\ref{Unipotentstabilizer}.

\subsection{Fixed points}\label{fixed points}  In this subsection $A$ and~$B$ vanish simultaneously at~$0$ (which is then  a fixed point for the local action  of the Lie algebra generated by $A$ and $B$).  By Lemma~\ref{exponential}, $A$ and $B$  are both linear in exponential coordinates.  Since the Lie algebra  generated by $A$ and $B$ is solvable, they are  simultaneously upper-triangular. Since  $B$ is a  nonzero  commutator, we may suppose that $B=x\indel{y}$. Up to replacing~$A$ by~$A+\lambda B$, we have, for some~$n\in\mathbf{R}^*$,
$$A= \frac{1}{n}x\del{x}+\left(\frac{1}{n}-1\right)y\del{y},\; B=x\del{y}.$$
These fields are linearly independent in the complement of~$\{x=0\}$. The vector fields that, in the half plane~$\{x>0\}$, commute with these two are linear combinations of
$$X=-\frac{1}{n}\left(x\del{x}+y\del{y}\right),\; Y=-x^{1-n} \del{y}.$$
They satisfy the relation~$[X,Y]=Y$. By Lemma~\ref{killingcentralizer}, the connection is of the form~(\ref{affinvcon}).

Solving for the Christoffel symbols, we have
\begin{multline*}\Gamma_{11}^1=-{\textstyle \frac{1}{n}}\epsilon x^{2n-3}y^2+2\gamma x^{n-2}y-(n\alpha+1)x^{-1},\\
\Gamma_{11}^2=-{\textstyle \frac{1}{n}}\epsilon x^{2n-4}y^3+(2\gamma-\phi)x^{n-3}y^2+(1-2n-n\alpha+2n\delta)x^{-2}y-\beta n^2 x^{-n-1},\\
\Gamma_{12}^1={\textstyle \frac{1}{n}}\epsilon x^{2n-2}y-\gamma x^{n-1},\;
\Gamma_{12}^2={\textstyle \frac{1}{n}}\epsilon x^{2n-3}y^2+(\phi-\gamma)x^{n-2}y+(n-1-n\delta)x^{-1},\\
\Gamma_{22}^1=-{\textstyle \frac{1}{n}}\epsilon x^{2n-1},\;
\Gamma_{22}^2=-{\textstyle \frac{1}{n}}\epsilon x^{2n-2}y-\phi x^{n-1}.
\end{multline*}

Since~$\nabla$ is supposed to be analytic  in the  neighborhood of the origin, we must have~$\alpha=-1/n$ and~$\delta=1-1/n$.

If one of the Christoffel symbols does not vanish (if~$\nabla$ has a chance of not being flat), we must have~$n\in\frac{1}{2}\mathbf{Z}$. If~$n<2$, then all of the Christoffel symbols must vanish except, possibly, for~$\Gamma_{11}^2$, which will be a function of~$x$. In this case the curvature of~$\nabla$ vanishes and~$\nabla$ is flat. We are left with the case~$n \geq  2$ which implies, in particular, that~$\beta=0$. If~$\epsilon$, $\phi$ and~$\gamma$  vanish simultaneously, all the Christoffel symbols vanish and the connection is flat.  Leaving the flat  case aside, $\nabla$ is analytic at the origin in the following cases: 
\begin{itemize}
\item $n=2$, $\phi=2\gamma$.
\item $n\in\mathbf{Z}$, $n \geq 3$;
\item $n\notin\mathbf{Z}$, $\gamma =0$, $\phi=0$ and $n\geq 5/2$;
\end{itemize}

In the first case, we have a connection of Type~$\mathrm{III}$. According to Proposition~\ref{typei-ii}, $\mathfrak{K}(\nabla)\approx \mathfrak{sl}(2,\mathbf{R})$ and is, in the chosen coordinates, the Lie algebra of divergence-free linear vector fields. 

In the other two, the connections are  of Type~$\mathrm{II}^0(n)$.

\subsection{One-dimensional orbits} We will now analyze the case when the rank at the origin of the Killing algebra generated by $A$ and $B$ is one.

\subsubsection{Semisimple stabiliser}\label{Semi-simplestabilizer} We prove now part (\ref{pa2}) of Proposition~\ref{propaffine}. We will now suppose that~$B$ does not vanish at the origin, but that~$A$ does. In exponential coordinates at the origin, the Killing vector field~$A$ is linear. Let~$B=\sum_{i=0}^\infty B_i$ where~$B_i$ is a homogeneous polynomial vector field of degree~$i$. By homogeneity, we have~$[A,B_i]=B_i$ and, in particular, $[A,B_0]=B_0$. Thus, $-1$ is an eigenvalue of~$A$ and $B_0(0)=B(0)$ is the corresponding eigenvector. In particular, $A$ has real eigenvalues. It may or may not be diagonalizable.\\

\paragraph{Diagonalizable case}  If~$A$ is diagonalizable, we may suppose that it is of the form~$A=\lambda x\indel{x}-y\indel{y}$, with $\lambda\in\mathbf{R}$ and that~$B_0=\indel{y}$. For the sake of computations, let $B=y^{-1}xf(x,y)\indel{x}+g(x,y)\indel y$ with~$g$ an analytic function such that $g(0)= 1$ and~$f$ a meromorphic function such that~$B$ is analytic (such that~$xf$ is an analytic function which admits $y$ as a factor). If
$$f=\sum_{i=-1, j=1} a_{ij}x^iy^j,\; g=\sum_{i=0, j=0} b_{ij}x^iy^j$$ the equation~$[A,B]=B$ reads, at the formal level, $(\lambda i-j)a_{ij}=0$, $(\lambda i-j)b_{ij}=0$ with~$i$ and~$j$ in the corresponding range. There are four possibilities
\begin{itemize}
\item $\lambda\notin \mathbf{Q}$. This implies that $f\equiv 0$, $g\equiv 1$ and thus
\begin{equation}\label{una}A=\lambda x\del{x}-y\del{y} , \; B=\del{y}.\end{equation}
\item $\lambda\in\mathbf{Q}$, $\lambda <0$. This implies that~$g\equiv 1$, that~$\lambda\in\mathbf{Z}$ and that~$f=cx^{-1}y^{-\lambda}$ for some~$c\in\mathbf{R}$ (we may suppose~$c=1$ up to rescaling~$x$). In the coordinates~$(\lambda x+y^{-\lambda},y)$, we still have~(\ref{una}).
\item $\lambda=0$, $f\equiv 0$, $g=g(x)$, but in this case~$A$ and~$B$ are linearly independent nowhere.
\item $\lambda=p/q$ for relatively prime and positive integers~$p$ and $q$. This implies that~$f$ and~$g$ are functions of~$u=x^qy^p$, with~$f(0)=0$. If~$h$ and~$k$ are functions of~$u$,  with~$k(0)=1$, $h(0)=1$, satisfying the differential equations
$$uh'=-\frac{fh}{qf+pg},\; uk'=\frac{1-gk}{qf+pg},$$
then, in the coordinates~$(xh,yk)$, we get  the same normal form for $A$ and $B$ as in~(\ref{una}). 
\end{itemize}

The existence of solutions to the last equations is guaranteed by the following Lemma, whose proof may be found in~\cite[\S 12.6]{ince}.

\begin{lemma}[Briot-Bouquet]\label{briotbouquet} For $i=1,2$, let~$F_i(x_1,x_2,t)$ be analytic functions in the neighborhood of~$p=(a_1,a_2,0)\in\mathbf{R}^3$ such that~$F_i(a_1,a_2,0)=0$. If neither of the eigenvalues of~$\partial F_i/\partial x_j|_{p}$ is a strictly positive integer, then the system of ordinary differential equations~$tx_i'=F_i(x_1,x_2,t)$ has an analytic solution~$(x_1(t),x_2(t))$ with~$x_i(0)=a_i$.\end{lemma}

We may thus suppose that the vector fields $A$ and $B$ are given by the expressions~(\ref{una}) with~$\lambda\neq 0$. We introduce the notation $\lambda=1/n$, $n\in\mathbf{R}^*$. The vector fields $A$ and $B$ are linearly independent in the half plane~$\{x>0\}$.  The vector fields in this half-plane that simultaneously commute with them are linear combinations of
$$X=-\frac{1}{n}x\del{x},\;Y=-x^{-n}\del{y}.$$

As in the previous case, $\nabla$ admits the expression~(\ref{affinvcon}), with respect to the moving frame $(X,Y)$.
Solving, as before, for the Christoffel symbols, we have
\begin{multline*}\Gamma_{11}^1=-(n\alpha+1)x^{-1},\; \Gamma_{11}^2=-n^2\beta x^{-n-2},\;
 \Gamma_{12}^1=-\gamma x^n,\\
\Gamma_{12}^2=n(1-\delta)x^{-1},\; \Gamma_{22}^1=-{\textstyle \frac{1}{n}}\epsilon x^{2n+1},\; \Gamma_{22}^2=-\phi x^n.\end{multline*}

Since $\Gamma_{ij}^k$ are analytic functions at the origin, we  conclude that $\delta=1$ and~$\alpha=-1/n$, which implies $\Gamma_{11}^1=\Gamma_{12}^2 \equiv 0$. In fact we get more: either~$\Gamma_{ij}^k\equiv 0$, for all $1 \leq i,j,k \leq 2$ or~$n\in\frac{1}{2}\mathbf{Z}$. If~$n<-1/2$, then all of the Christoffel symbols must vanish except, possibly, for~$\Gamma_{11}^2$, which will be a function of~$x$. If~$n=-1/2$, all of the Christoffel symbols must vanish except for~$\Gamma_{22}^1$, which will be constant. In both cases the curvature of~$\nabla$ vanishes. If~$n=0$ the connection cannot be analytic.

We thus conclude that~$n>0$ and, in consequence, that~$\beta=0$. If~$\epsilon$, $\gamma$ and~$\phi$ all vanish,  the connection is flat, so we will suppose that one of these constants does not vanish. We have a non-flat analytic connection in the following situations:
\begin{itemize}
\item $n\in\mathbf{Z}$, $n\geq 1$,
\item $n\in\frac{1}{2}\mathbf{Z}$, $n\geq\frac{1}{2}$,  $\gamma=0$, $\phi=0$.
\end{itemize}

The connection induced in the right half-plane is an invariant connection of Type~I, with respect to the marking $X$.
By Proposition~\ref{typei-ii}, the Killing algebra is two-dimensional (and hence the one generated by~$A$ and~$B$) except in the case where~$n=1$ and~$\phi=-\gamma$, $\epsilon=-\gamma^2$,  with $\gamma \neq 0$, where it has dimension three. In this last case we have  the extra Killing field $2(xy+\gamma^{-1})\indel{x}-y^2\indel{y}$ and~$\nabla$ is thus locally homogeneous (we do not consider this last situation, since $\nabla$ is supposed not to be locally homogeneous). We thus find the connections of Type~$\mathrm{I}$ of Theorem~\ref{thm-local}. \\

\paragraph{Non-diagonalizable case} We may suppose that~$A=(- x+y)\indel{x}-y\indel{y}$ and that~$B_i=f_i\partial/\partial x+g_i\partial/\partial y$ with~$f_i$ and~$g_i$ homogeneous polynomials of degree~$i$. Up to multiplying~$B$ by a constant, we assume~$f_0=1$. The equation~$[A,B_i]=B_i$ splits (via Euler's relation) into the equations
$y  \partial g_i/\partial x=ig_i$, $y \partial f_i/\partial x=if_i$.
Hence, for every~$k\in\mathbf{N}$, $y \partial^k g_i/\partial x^k=i^kg_i$. Since~$g_i$ and~$f_i$ are polynomials, $g_i\equiv 0$, for all $i$,  and~$f_i\equiv 0$, for all ~$i\neq 0$. Thus,
$$A=(y-x)\del{x}-y\del{y},\; B=\del{x}.$$
They are linearly independent in the complement of~$\{y=0\}$. The space of vector fields that commute with these two in~$\{y>0\}$ is generated by
$$X= -y\log y\del{x}+y\del{y},\; Y=-y\del{x}.$$

The connection $\nabla$ admits the expression~(\ref{affinvcon}) with respect to $(X,Y)$.
Solving for the Christoffel symbols, 
\begin{multline*}\Gamma_{11}^1=-\frac{\phi}{y}-\frac{\epsilon}{y}\log y,\; \Gamma_{11}^2= \frac{\epsilon}{y} ,\;
 \Gamma_{12}^1= \frac{\delta-1}{y}+\frac{\gamma-\phi}{y}\log y-\frac{\epsilon}{y}\log^2 y,\\
\Gamma_{12}^2=-\frac{\gamma}{y}+\frac{\epsilon}{y}\log y ,\; \Gamma_{22}^1=\frac{\beta-1}{y}+\frac{2\delta-1-\alpha}{y}\log y +\frac{2\gamma-\phi}{y}\log^2 y -\frac{\epsilon}{y}\log^3 y ,\\ \Gamma_{22}^2= \frac{\alpha-1}{y}-\frac{2\gamma}{y}\log y+\frac{\epsilon}{y}\log^2 y.\end{multline*}
If the connection extends analytically to the origin, $\Gamma_{ij}^k\equiv 0$ and the connection is flat. This finishes the proof of item (\ref{pa2}) in Proposition~\ref{propaffine}. 

\subsubsection{Unipotent stabilizer}\label{Unipotentstabilizer}  We deal now with part (\ref{pa3}) of Proposition~\ref{propaffine}. We  suppose that~$B$ vanishes at~$0$, but that~$A$ does not. We may linearize~$B$ and, since it must vanish along the common one-dimensional orbit of~$A$ and~$B$, it must have one vanishing eigenvalue. Supposing that $B$ vanishes along~$x=0$, either
$B=x\indel{x}$ or~$B=x\indel y$. If~$A=xf\indel{x}+h\indel{y}$, the bracket relations imply, in the first case, $xf_x=-1$, and thus~$f(x)=-\log x+c$ and~$A$ cannot be analytic. In the second one, we have the relations~$f_y=0$ and~$h_y=f-1$ and thus~$f=f(x)$ and~$h=(f-1)y+g(x)$ with~$f\not\equiv 0$, $g(0)\neq 0$ (in particular, $f$ is analytic at the origin). Hence,
\begin{equation}\label{abcasfinal}A=xf(x)\del{x}+(y[f(x)-1]+g(x))\del{y},\; B=x\del y.\end{equation}
They are linearly independent in the complement of~$x=0$.
The vector fields that commute with these two in~$\{x>0\}$ are linear combinations of
$$X=-\left(xf(x)\del{x}+\left[yf(x)+u(x)\right]\del{y}\right),\;Y=\frac{1}{v(x)}\del{y},$$
where~$u$ and~$v$ are non-zero solutions to the differential equations
\begin{equation}\label{diffeqnsuv}xu'=xg'-g+\left(\frac{f-1}{f}\right)u,\;xv'=-\left(\frac{f-1}{f}\right)v,\end{equation}
defined in~$\{x>0\}$. By solving for the Christoffel symbols and using the above relations,
\begin{multline*}
\Gamma_{11}^1=-\frac{\alpha}{xf}-\frac{1}{x}-\frac{f'}{f}-\frac{2\gamma v}{xf}\left(yf+u\right)-\frac{\epsilon v^2}{xf}\left(yf+u\right)^2, \\
\Gamma_{11}^2=\frac{[2\delta-\alpha-2+f]}{x^2f^2}(yf+u)+\frac{\beta}{x^2f^2v}+\frac{[\phi-2\gamma]v}{x^2f^2}(yf+u)^2-\frac{\epsilon  v^2}{x^2f^2}(yf+u)^3-\frac{yf'+u'}{xf},\\
\Gamma_{12}^1=\gamma v+\epsilon v^2(yf+u),\\
\Gamma_{12}^2=\frac{\epsilon v^2}{xf}(yf+u)^2+\frac{[\gamma-\phi]v}{xf}(yf+u)-\frac{\delta}{xf}+\frac{1}{xf}-\frac{1}{x} ,  \\
\Gamma_{22}^1=-\epsilon xfv^2,\;
\Gamma_{22}^2=\phi v-\epsilon v^2(yf+u).
\end{multline*}

We will now prove that $f(0) \neq 0$. If~$f(0)=0$ then, from equation~(\ref{diffeqnsuv}), $v$ has an essential singularity at the origin, for its logarithmic derivative will have more than simple poles. Since~$\Gamma_{11}^1$ is analytic  at  the origin and  the coefficient of~$y^2$ in its expression contains $v^2$, this coefficient must vanish and thus $\epsilon =0$. Also  the coefficient of~$y$ in $\Gamma_{11}^1$ must vanish, and thus  $\gamma=0$. We must have~$\alpha=0$ for otherwise~$\Gamma_{11}^1$ would have pole of order greater than one at the origin. Finally, $\Gamma_{11}^1=-1/x-f'/f$ which has a pole at the origin since~$f$ is analytic: a contradiction.

Hence, if the connection extends analytically to the origin, $f(0)\neq 0$. We will now find new coordinates where the expressions of~$X$ and~$Y$ simplify.

Choose a coordinate~$z=x\xi(x)$ in which  the vector field~$xf(x)\partial/\partial x$ reads~$n^{-1}z\partial/\partial z$ for some~$n\in\mathbf{R}^*$. In the coordinates~$(x\xi(x),y\xi(x))$, which preserve the form of $X$ and $Y$, we have~$f(x)\equiv n^{-1}$. Hence, from equation~(\ref{diffeqnsuv}), $v(x)=-x^{n-1}$. We thus have
$$X=-\left( \frac{1}{n} x\del{x}+\left[ \frac{1}{n} y+u(x)\right]\del{y}\right),\;Y=-x^{1-n} \del{y}.$$
We will simplify even more the normal form of $X$ and $Y$. For this, we will first prove that $n\in\frac{1}{2}\mathbf{Z}$, $n\geq 2$.

Assume, by contradiction, that $v^2/x^2$ is \emph{not} analytic. From the coefficient of~$y^3$ in~$\Gamma_{11}^2$, $\epsilon=0$; from the coefficient of~$y$ in~$\Gamma_{12}^2$  and in~$\Gamma_{11}^1$, $\gamma=0$ and~$\phi=0$; from the analyticity of~$\Gamma_{11}^1$ and~$\Gamma_{12}^2$, $\alpha=-1/n$, $\delta=1-1/n$. These conditions imply that all the Christoffel symbols vanish except (possibly) for
$\Gamma_{11}^2$ which will be a function of~$x$. Thus  the connection $\nabla$  is flat: a contradiction. We must conclude that~$v^2/x^2$ is analytic and thus that~$n\in\frac{1}{2}\mathbf{Z}$, $n\geq 2$.

We apply  Lemma~\ref{briotbouquet} and consider $\psi(x)$  a solution to the differential equation
$x\psi'=[1-n]\psi-n[g(x)-g(0)]$ that
vanishes at the origin (the equation has a solution, for~$n>1 $ implies $1-n< 0$).

Now we replace  replace~$y$ by~$y+\psi(x)$. The vector field~$B$ in~(\ref{abcasfinal}) is preserved ; the vector field~$A$ changes to one where~$g$ is constant. Up to multiplying~$y$ by a constant,  we will suppose that~$g\equiv 1/n-1$. The differential equation~(\ref{diffeqnsuv}) for~$u$ becomes $xu'=\left(1-1/n\right)(1-n u)$. The general solution of this equation is~$u(x)=1/n+cx^{1-n}$. Hence, up to the addition of a multiple of~$Y$ to~$X$, we can suppose that~$u\equiv 1/n$. We thus have
$$A=   \frac{1}{n} \left( x\del{x}-(n-1)(y+1)\del{y} \right),\; B=x\del y$$
$$X=-   \frac{1}{n} \left(  x\del{x}+ (y+1)\del{y}\right),\;Y=-x^{1-n} \del{y}.$$
Once again, $\nabla$ has the expression~(\ref{affinvcon}) with respect to $(X,Y)$ and using the last expression of $X$ and $Y$, we find the Christoffel coefficients:
\begin{multline*}
\Gamma_{11}^1=-(n\alpha+1)x^{-1}+ 2\gamma x^{n-2} \left(y +1\right)-{\textstyle{\frac{1}{n}}}\epsilon x^{2n-3} \left(y +1\right)^2, \\
\Gamma_{11}^2=-n^2 \beta x^{-n-1}- [2(n-n\delta-1)+(n\alpha+1)] x^{-2}(y +1)-[\phi-2\gamma] x^{n-3} (y +1)^2-{\textstyle{\frac{1}{n}}}\epsilon  x^{2n-4}(y +1)^3,\\
\Gamma_{12}^1=-\gamma x^{n-1}+{\textstyle{\frac{1}{n}}}\epsilon x^{2n-2}(y+1),\\
\Gamma_{12}^2=(n-n\delta-1)x^{-1}-[\gamma-\phi] x^{n-2} (y +1)+{\textstyle{\frac{1}{n}}}\epsilon  x^{2n-3}(y +1)^2  \\
\Gamma_{22}^1=-{\textstyle{\frac{1}{n}}}\epsilon x^{2n-1},\;
\Gamma_{22}^2=-\phi x^{n-1}-{\textstyle{\frac{1}{n}}}\epsilon x^{2n-2}(y + 1).
\end{multline*}
For~$\Gamma_{11}^2$ and~$\Gamma_{11}^1$ to be analytic at the origin, $\beta=0$ and $\alpha=-1/n$. From the coefficient of~$x^{-2}$ in~$\Gamma_{11}^2$, $\delta=1-1/n$. If~$\epsilon$, $\gamma$ and~$\phi$ vanish simultaneously, the connection is flat. Leaving this case aside, the connection extends analytically to the origin in the following cases (recall that we have already established $n\in\frac{1}{2}\mathbf{Z}$, $n\geq 2$)
\begin{itemize}
\item $n=2$, $\phi=2\gamma$.
\item $n\in\mathbf{Z}$, $n\geq 3$
\item $n=m/2$, $m \in\mathbf{Z}$, $m\geq 5$ odd, $\gamma =0$, $\phi=0$.
\end{itemize}
The invariant connection in~$\mathrm{Aff}_0(\mathbf{R})$ associated to the restriction of this connection in~$\{x>0\}$ is of Type~II. From Proposition~\ref{typei-ii}, $\mathfrak{K}(\nabla)$ has dimension two (and  is generated by~$A$ and~$B$) except in the first case, where it has dimension three (in this last case~$(y+1)\indel{x}\in\mathfrak{K}(\nabla)$ and the connection is locally homogeneous). This eliminates the case $n=2$. These vector fields and connections are the germs at $(0,1)$ of those we found in the proof of part (\ref{pa1}) in Proposition~\ref{propaffine} (section~\ref{fixed points}). The connections are thus of Type~$\mathrm{II}^1$ .

\subsection{Local isometries and normal forms} So far we have proved that any germ of connection satisfying the hypothesis of Theorem~\ref{thm-local}, such that one of the two-dimensional algebras guaranteed by Proposition~\ref{dim2} is affine, is indeed the germ of a connection of Type~$\mathrm{I}$, $\mathrm{II}$ or~$\mathrm{III}$ in the neighborhood of some point~$O\in\mathbf{R}^2$ (the point~$(0,1)$ for the connections of Type~$\mathrm{II}^1$, the origin for the other cases). In order to get normal forms we must investigate the parameters that yield isometric connections.\\

The isomorphism type of the Killing Lie algebra of a connection is an isometry invariant. Our discussion splits naturally in two cases:

\subsubsection{Types~$\mathrm{I}$, $\mathrm{II}^0$ and $\mathrm{II}^1$ } By Proposition~\ref{propaffine}, the isometry class of a connection~$\nabla$ cannot belong simultaneously to two of the Types~$\mathrm{I}$, $\mathrm{II}^0$, $\mathrm{II}^1$. In all these cases the rank of~$\mathfrak{K}(\nabla)$ is one along~$\{x=0\}$ and two on its complement. In restriction to each connected component of the complement, $\nabla$ is locally modeled by a left-invariant connection in the affine group. For the convenience of the reader we reproduce, in Table~\ref{table1}, the vector fields associated to the normal forms of Theorem~\ref{thm-local} (recall that for the connections of Type $\mathrm{II}^1$, one should consider germs at $(0,1)$).\\

\begin{table}
\begin{tabular}{|c|c|c|c|c|}
\hline Type & $A$ & $B$ & $X$ & $Y$ \\ \hline
$\mathrm{I}(n)$ &  ${\textstyle \frac{1}{n}} x\del{x}-y\del{y}$ & $\del{y}$ &
$-{\textstyle \frac{1}{n}}x\del{x}$ & $-x^{-n}\del{y}$  \\ \hline
$\mathrm{II}(n)$ & ${\textstyle \frac{1}{n}}\left( x\del{x}+[1-n]y\del{y}\right)$ & $x\del{y}$ &
 $-{\textstyle \frac{1}{n}}\left(x\del{x}+y\del{y}\right)$ & $-x^{1-n}\del{y}$ \\ \hline
\end{tabular}
\caption{Killing vector fields ($A$ and~$B$) and their centralizers ($X$ and~$Y$) in~$\{x>0\}$  for the connections of Theorem~\ref{thm-local} whose Killing algebra is isomorphic to~$\mathfrak{aff}(\mathbf{R})$ ($n\in\frac{1}{2}\mathbf{Z}$, $n\geq \frac{1}{2}$ for Type~$\mathrm{I}$ and $n\geq \frac{5}{2}$ for Type~$\mathrm{II}$).}\label{table1}
\end{table}

We will begin by proving that, within the connections of Type~$\mathrm{I}$, the value of~$n$ is an isometry invariant and that the same is true for the connections of Types~$\mathrm{II}^0$ and $\mathrm{II}^1$. 

For Type~$\mathrm{I}(n)$, $A$ generates the space of Killing vector field vanishing at~$0$. It is normalized by~$[A,B]=B$. Its eigenvalue at the origin in the direction transverse to~$\{x=0\}$ is~$1/n$.

For Type~$\mathrm{II}^0(n)$, the semisimple vector fields are those of the form~$A+\lambda B$. All of them vanish at the origin and the eigenvalue in the direction transverse to~$\{x=0\}$ is~$1/n-1$.

For Type~$\mathrm{II}^1(n)$, the vector field~$X$ is the only analytic vector field in the centralizer of~$\mathfrak{K}(\nabla)$ such that, on either side of~$\{x=0\}$, there exists a vector field~$Y$ on the centralizer of~$\mathfrak{K}(\nabla)$ such that~$[X,Y]=Y$. The most general semisimple vector field is~$A+\lambda B$ and has, in the neighborhood of~$(0,1)$, the  primitive first integral~$h=x(y-\lambda x)^{1/(n-1)}$. We have~$X\cdot h=(1-n)^{-1}h$. The eigenvalue~$(1-n)^{-1}$ is intrinsically attached to~$\nabla$.\\

If a diffeomorphism maps a connection~$\nabla$ of Type~$\mathrm{I}(n)$ to a connection~$\nabla'$ of the same Type, it must map~$\mathfrak{K}(\nabla)$ to~$\mathfrak{K}(\nabla')$ and thus it must actually preserve~$\mathfrak{K}(\nabla)$, which is independent of the parameters~$(\gamma,\epsilon,\delta)$. The same is true for connections of Types~$\mathrm{II}^0(n)$ and $\mathrm{II}^1(n)$. We will study, in all cases,  the group of germs of diffeomorphisms fixing~$O$ and preserving~$\mathfrak{K}(\nabla)$. \\

Any germ of diffeomorphism~$g:(\mathbf{R}^2,O)\to(\mathbf{R}^2,O)$ that preserves~$\mathfrak{K}(\nabla)$ should preserve the line~$\{x=0\}$, where the rank of this algebra is one. It may do so either by preserving or by exchanging the connected components of its complement. It may  preserve or reverse the orientation. Hence, within the group of germs of automorphisms of~$\mathfrak{K}(\nabla)$, the subgroup of those preserving the orientation and the half-plane~$\{x=0\}$ is a normal subgroup of index one, two or four.\\

If~$g$ preserves the half-plane~$\{x>0\}$, then it must preserve the centralizer of~$\mathfrak{K}(\nabla)$ within the half-plane, the Lie algebra generated by~$X$ and~$Y$. Since, within this algebra, $X$ is the only element that extends analytically to the origin and  is normalized by~$[X,Y]=Y$, we must have~$g_*X=X$.  From~(\ref{actionadjconn}), $g$ has the effect of multiplying~$Y$ by some constant, positive, if~$g$ preserves orientation, negative, if it does not. 

We claim that, in the case when~$g$ preserves orientation, the positive constant is arbitrary. For the connections of Types~$\mathrm{I}$ and~$\mathrm{II}^0$, the flow of the vector field~$X$ fixes~$O$, preserves~$\mathfrak{K}(\nabla)$, preserves~$X$ and multiplies~$Y$ by an arbitrary positive constant. 
For the connections of Type~$\mathrm{II}^1$, the vector field~$A+(1-n)X$ vanishes at~$O$, its flow preserves~$\mathfrak{K}(\nabla)$ and~$X$ and multiplies~$Y$ by an arbitrary positive constant. Hence,  the action upon the parameters of the connection of these orientation-preserving mappings is
\begin{equation}\label{actionparameteres} (\gamma,\epsilon,\phi)\mapsto (\mu \gamma, \mu^2\epsilon,\mu\phi) \end{equation}
for~$\mu>0$.

Let us now consider the case where~$g$  preserves the half-plane~$\{x>0\}$,  but not the orientation. We claim that this cannot happen  for the family of connections of Type~$\mathrm{II}^1$ and that the above negative constant is arbitrary for the connections of Types~$\mathrm{I}$ and~$\mathrm{II}^0$.  Let
\begin{equation}\label{miroirtau}\rho(x,y)=(x,-y).\end{equation}
For the connections of Types~$\mathrm{I}$ and~$\mathrm{II}^0$, $\rho(O)=O$, $\rho$ preserves the half-plane~$\{x>0\}$ and satisfies~$\rho_*X=X$, $\rho_*Y=-Y$. In these cases, the action upon the parameters of the connection is~(\ref{actionparameteres}) for~$\mu=-1$. This, combined with the previous results, shows that the negative constant is arbitrary.  If~$\nabla$ is a connection of Type~$\mathrm{II}^1$ and if~$g:(\mathbf{R}^2,O)\to(\mathbf{R}^2,O)$ preserves $\mathfrak{K}(\nabla)$ and the half-plane~$\{x>0\}$, but does not preserve the orientation, by the previous arguments, we may suppose that~$g_*X=X$ and~$g_*Y=-Y$. The most general mapping doing this is given by
$$(x,y)\mapsto (e^t x,e^{(1-n)t}(2x-y)+s e^t x)$$
with~$s,t\in\mathbf{R}$. But if such a map fixes~$(0,1)$ then~$-e^{(1-n)t}=1$, which is impossible.\\

Let us now consider the case where~$g$ does not preserve neither the half-plane~$\{x>0\}$, nor the orientation. Let
\begin{equation}\label{miroirsig}
\sigma(x,y)=(-x,y).
\end{equation}
The vector field~$X$ is well-defined in a neighborhood of~$O$ and we have~$\sigma_*(X)=X$. 

If~$n\in\mathbf{Z}$, there is a natural extension of~$Y$ onto~$\{x<0\}$, that we will still call~$Y$. We have~$\sigma_*(X)=X$, $\sigma_*Y=(-1)^nY$ (for Type~$\mathrm{I}$) and~$\sigma_*Y=(-1)^{n+1}Y$ (for Types~$\mathrm{II}^0$ and~$\mathrm{II}^1$). In particular, $\sigma$ is an isometry for the connections of Type~$\mathrm{I}(n)$, when~$n$ is even, and for those of Type~$\mathrm{II}(n)$, if~$n$ is odd. For the opposite parities, the action upon the parameters of the connection is~(\ref{actionparameteres}) for~$\mu=-1$.

If~$n\notin\mathbf{Z}$, the vector field~$\sigma_*Y$ is well-defined in~$\{x<0\}$ and satisfies~$[X,\sigma_*Y]=\sigma_*Y$. We have~$\nabla_X X=\alpha X $, $\nabla_X (\sigma_*Y)=\delta (\sigma_*Y)$ and $\nabla_{\sigma_*Y} (\sigma_*Y)=-\epsilon X$. The transformation has the effect of changing the sign of~$\epsilon$ (in some sense, $\sigma$ acts upon the coefficients as~(\ref{actionparameteres}) for~$\mu=\sqrt{-1}$).\\

The last case, where~$g$ preserves the orientation, but not~$\{x>0\}$, is a composition of the two orientation-reversing cases (it takes place if and only if both of them take place).\\

The results for~$n\in\mathbf{Z}$ are summarized in Table~\ref{tab:isomnor}.

\begin{table}
\begin{center}
\begin{tabular}{|c|cc|ccc|} \hline
& \begin{tabular}{c}
preserves \\ $\{x>0\}$?
\end{tabular}
& \begin{tabular}{c}
preserves \\ orientation?
\end{tabular} & $\mathrm{I}(n)$ & $\mathrm{II}^0(n)$ & $\mathrm{II}^1(n)$ \\
\hline $1$ & yes & yes & $1$ & $1$ & $1$ \\
$\sigma$ & no & no &  $(-1)^n$ & $(-1)^{n+1}$ & $(-1)^{n+1}$\\
$\rho$ & yes & no & $-1$ & $-1$ & $\times$ \\
$\sigma\circ\rho$ & no & yes & $(-1)^{n+1}$ & $(-1)^{n}$ & $\times$ \\ \hline
\end{tabular}
\end{center}
\caption{Action upon the family of connections of some transformations preserving~$\mathfrak{K}(\nabla)$, in the case $n\in\mathbf{Z}$.}\label{tab:isomnor}
\end{table}

\begin{remark}\label{isomesp} We proved that~$\sigma$ is an isometry for the connections of Type~$\mathrm{I}(n)$, if~$n$ is even, and for the connections of Type~$\mathrm{II}(n)$, if~$n$ is odd. \end{remark}

\begin{remark}\label{twoside} For all the connections of Theorem~\ref{thm-local} there is a germ of isometry fixing~$O$ that does not preserve the half-plane~$\{x>0\}$.
 \end{remark}

\begin{remark} If~$n\notin\mathbf{Z}$, for the connections of Type~$\mathrm{I}$ and~$\mathrm{II}$ (seen as global connections in~$\mathbf{R}^2$), the left-invariant connection in the affine group induced by the restriction of~$\nabla$ to~$\{x>0\}$ is \emph{not} equivalent to the connection induced by the restriction to~$\{x<0\}$ (the sign of~$\epsilon$, which is an invariant, is not the same). These give examples of connections in~$\mathbf{R}^2$ (with polynomial Cristoffel symbols) where the isometry pseudogroup acts with an open orbit but not with a dense one. \end{remark}

\subsubsection{Type~$\mathrm{III}$}

The local diffeomorphisms~$\Phi$ preserving~$\mathfrak{K}(\nabla)$ are necessarily linear, for they must preserve the space of divergence-free linear vector fields on~$(\mathbf{R}^2,0)$.  They are generated by~$\mathfrak{K}(\nabla)$, by the homotheties (the flow of~$X$) and by~$\sigma$. Al these preserve~$X$ and the last two (which do not, a priori, preserve~$\nabla$) preserve~$Y$. The effect of these diffeomorphisms is to multiply~$Y$ by an arbitrary constant and we thus obtain the normal form for the connections of  Type~$\mathrm{III}$.

\section{The global models}\label{sec:const}

We will begin by proving parts~(\ref{main-1st}) and~(\ref{main-mod}) of Theorem~\ref{thm-global}, by constructing the connections~$\nabla_{n,m}$ in~$\mathbf{R}^2$ announced in the Theorem and by studying their isometries.

\subsection{The connection} The idea behind the construction of~$\nabla_{n,m}$ may be sketched as follows:
\begin{enumerate}
\item to view~$\mathbf{R}^2$ as the union of the sets $V_i=\{(x,y)|x\in (i,i+1) \}$, $\ell_i=\{(x,y)|x=i\}$, $i\in\mathbf{Z}$.
\item to consider the action of~$\widetilde{\sigma}(x,y)= (-x,y)$ and~$\widetilde{\phi}(x,y)=(x+2,y)$ on~$\mathbf{R}^2$
\item to construct a connection~$\nabla$ in a neighborhood of the closure of~$V_0$ whose germ at~$\ell_0$ is preserved by~$\widetilde{\sigma}$ and whose germ at~$\ell_1$ is preserved by~$\widetilde{\sigma}\circ\widetilde{\phi}^{-1}$.
\item to propagate the connection to~$\mathbf{R}^2$ by the group generated by~$\widetilde{\sigma}$ and~$\widetilde{\phi}$. 
\end{enumerate}

Within the left-invariant connections in~$\mathrm{Aff}_0(\mathbf{R})$, we will say that a marked connection  is of Type~$\Xi(n)$, if it is of Type~$\mathrm{I}_0(n)$ and~$n$ is even, or if it is of Type~$\mathrm{II}^1_0(n)$ and~$n$ is odd.

Let~$(n,m)\in \mathbf{Z}^2$, $n,m\geq 2$. Let~$\nabla_0$ be a left-invariant connection in~$\mathrm{Aff}_0(\mathbf{R})$ such that~$X_0$ and~$X_0-Y_0$ are markings of~$\nabla$, such that~$(\nabla_0,X_0-Y_0)$ is of Type~$\Xi(m)$ and that~$(\nabla_0,X_0)$ is of Type~$\Xi(n)$. Such a connection \emph{exists and is unique} by formula~(\ref{iftwomarkings}). Let~$(\gamma,\phi, \epsilon)=(\gamma_0,\phi_0, \epsilon_0)$ be the corresponding parameters.

\begin{lemma}\label{nonvanishingparams} The parameters~$\gamma_0$ and~$\phi_0$ do not vanish simultaneously.\end{lemma}

\begin{proof}If~$\gamma_0$ and~$\phi_0$ vanished simultaneously, we would have, from~(\ref{deuxtypes}), $\delta+\delta'-1-\alpha=0$ and~$1-2\delta+\alpha'=0$. If~$(\nabla,X_0)$ were of Type~$\mathrm{I}$, we would have~$\delta=1$ which would then imply~$\alpha'=1$, an impossibility. If~$(\nabla,X_0)$ were of Type~$\mathrm{II}$, we would have~$\delta=1+\alpha$ and this would imply~$\delta'=0$, forcing~$(\nabla,X_0+\lambda Y_0)$ to be of type~$\mathrm{II}$, this is, $\delta'=1+\alpha'$ and thus $\alpha'=-1$. But this contradicts the fact that~$m\geq 2$. \end{proof}

Let~$V_0=\{(u,v)\in\mathrm{Aff}_0(\mathbf{R}), 0< v/u<1 \}$. It is invariant by~$A_0$ and is bounded by the common orbits of $A_0$ and the two special markings of~$\nabla_0$ just defined. Equivalently, $V_0$ is the set where  
\begin{equation}\label{orient}\frac{A_0\wedge(X_0-Y_0)}{A_0\wedge X_0}=1-\frac{u}{v}\end{equation}
is negative. Let~$\beta:\mathrm{Aff}_0(\mathbf{R})\to \mathrm{Aff}_0(\mathbf{R})$ be the involutive mapping given by
\begin{equation}\label{voltear}\beta(u,v)=(u,u-v).\end{equation}
The image of~$\nabla_0$ under~$\beta$ is the unique left-invariant connection~$\nabla_0'$ in~$\mathrm{Aff}_0(\mathbf{R})$ such that~$X_0$ and~$X_0-Y_0$ are markings of~$\nabla_0'$, such that~$(\nabla_0',X_0-Y_0)$ is of Type~$\Xi(n)$ and such that~$(\nabla_0',X_0)$ is of Type~$\Xi(m)$. Moreover, $\beta$ preserves~$V_0$ and the vector field~$A_0$. In particular, if~$n=m$, $\beta:V_0\to V_0$ is an orientation-reversing isometry of~$\nabla_0$. Within the affine group, $\beta$ is the inner automorphism of~$(-1,0)$ followed by the right translation by~$(1,1)$.

We will add a boundary component to~$V_0$ and extend~$\nabla_0$ analytically  to it:
\begin{description}
\item[Case~1, $n$ is odd] Let~$\nabla$ be a connection of type~$\mathrm{II}^1(n)$ in~$\mathbf{R}^2$ with the parameters~$(\gamma,\phi,\epsilon)=(\gamma_0,\phi_0, \epsilon_0)$. Consider the embedding~$\Psi:V_0\to \mathbf{R}^2$ given by
$$\Psi(u,v)=\left(\frac{1}{\sqrt[n]{u}},\frac{v}{\sqrt[n]{u}}\right)=(x,y).$$
It maps~$X_0$ to $X$, $Y_0$ to~$Y$ and~$A_0$ to~$A$. In consequence, it maps~$\nabla_0$ to~$\nabla$. The first integral~$v/u$ of~$A_0$ is mapped to the first integral~$yx^{n-1}$ of~$A$ and thus~$\Psi(V_0)=\{(x,y)|\, 0<x, 0<yx^{n-1}<1\}$. Let~$\ell=\{(x,y)\in \mathbf{R}^2, x=0, y>0\}$.

\item[Case 2, $n$ is even] Let~$\nabla$ be a connection of type~$\mathrm{I}(n)$ with parameters~$(\gamma,\phi,\epsilon)=(\gamma_0,\phi_0, \epsilon_0)$. Consider the embedding~$\Psi :V_0\to \mathbf{R}^2$ given by
$$\Psi(u,v)=\left(\frac{1}{\sqrt[n]{u}},v-1\right)=(x,y).$$
The vector fields $X_0$ and $Y_0$ are mapped, respectively, to the vector fields~$X$ and $Y$ and thus, $\Psi$ maps~$\nabla_0$ to~$\nabla$. The vector field~$A_0$ is mapped to~$A-B$. The first integral~$v/u$ of~$A_0$ is mapped  to the first integral~$(1+y)x^{n}$ of~$A-B$ and thus~$\Psi(V_0)=\{(x,y)|\, x>0, 1>(1+y)x^n>0\}$. Let~$\ell=\{(x,y)|\, x=0, y>-1\}$.\end{description}

The set~$\Psi(V_0)\cup\ell$ is naturally a manifold-with-boundary and gives an embedding of~$V_0$ into a manifold-with-boundary~$V_0\cup\partial^+V_0$ to which~$\nabla_0$ extends. By exchanging the roles of~$n$ and~$m$, via~$\beta$, the above construction may be repeated in order to construct another boundary component~$\partial^-V_0$ of~$V_0$ and to extend analytically the connection to it.

We thus have a manifold-with-boundary~$\overline{V}_0$ whose interior is~$V_0$ and whose boundary components are~$\partial^+V_0$ and~$\partial^-V_0$. There is an analytic connection~$\nabla_0$ on~$\overline{V}_0$.

For~$j\in \mathbf{Z}$, let~$\overline{V}_j=\overline{V}_0$ together with a tautological mapping~$\pi_j:\overline{V}_0\to \overline{V}_j$. Let~$\nabla_j$ be the connection in~$\overline{V}_j$ such that~$\pi_j$ is an isometry. We will now glue~$\overline{V}_0$ and~$\overline{V}_1$ along~$\partial^+V_0$ and~$\partial^+V_1$ while gluing~$\nabla_0$ to~$\nabla_1$. Consider the immersions~$\Psi:(V_0\cup\partial^{+}V_0)\to\mathbf{R}^2$ and~$\sigma\circ\Psi\circ\pi_1^{-1}:(V_1\cup\partial^+V_1)\to\mathbf{R}^2$, for the mapping~$\sigma$ defined in~(\ref{miroirsig}). By identifying the points that have the same image, we obtain an identification between~$\partial^{+}V_0$ and~$\partial^+V_1$ that produces a submanifold~$\ell_1$ of the gluing. For the connections of Type~$\Xi(n)$, $\sigma$ is an isometry of~$\nabla$ (Remark~\ref{isomesp}) and the connections~$\nabla_0$ and~$\nabla_1$ agree after the gluing.  We can proceed in exactly the same way to identify~$\partial^+V_{2k}$ with~$\partial^+V_{2k+1}$.  By exchanging the roles of~$n$ and~$m$, via $\beta$,  we can now glue~$\partial^-V_{2k}$ with~$\partial^-V_{2k-1}$.

In this way we obtain a manifold~$M$, diffeomorphic to~$\mathbf{R}^2$, endowed with a connection~$\nabla_{n,m}$. This is the connection of item~(\ref{main-1st}) of Theorem~\ref{thm-global}.

\subsection{The Isometries}
We will  now study the global isometries of the connection~$\nabla_{m,n}$ just defined.

The vector field~$A$ is globally well-defined in~$M$ for~$\sigma_* A =A$. It is complete, since it preserves~$V_i$ and~$\ell_i$, for every~$i$, and is complete in restriction each of them. The flow of~$A$ gives thus a action of~$\mathbf{R}$ by isometries.

There is a unique orientation-preserving (tautological) isometry~$\phi:M\to M$ such that~$\phi(V_j)=V_{j+2}$ and such that~$\phi\circ\pi_{-2}:V_0\to V_0$ is the identity. This isometry gives an action of~$\mathbf{Z}$ on~$M$ by orientation-preserving isometries. The flow of~$A$ commutes with~$\phi$.

If~$n=m$ there is, moreover, a unique orientation-preserving (tautological)  isometry~$\phi^\frac{1}{2}:M\to M$ such that~$\phi^\frac{1}{2}(V_j)=V_{j+1}$,  such that~$\phi^\frac{1}{2}|_{V_0}\circ\pi_{-1}$ is given by the mapping~$\beta$ of formula~(\ref{voltear}). The flow of~$A$ commutes with~$\phi^\frac{1}{2}$. We naturally have~$\phi^\frac{1}{2}\circ\phi^\frac{1}{2}=\phi$.

\begin{proposition} The group of orientation-preserving isometries of~$\nabla_{m,n}$ is isomorphic to~$\mathbf{Z}\times\mathbf{R}$. If~$m\neq n$, it is generated by~$\phi$ and by the flow of~$A$.  If~$m=n$, it is generated by~$\phi^\frac{1}{2}$ and by the flow of~$A$. 
\end{proposition}
\begin{proof} Let us begin by studying the group of global isometries of~$\nabla_0$ in~$V_0$ in order to prove that any orientation-preserving isometry of~$(V_0,\nabla_0)$ embeds into the flow of~$A_0$ and that~$V_0$ admits an orientation-reversing isometry if and only if~$n=m$.

Let~$f:V_0\to V_0$ be an isometry of~$\nabla_0$. From Proposition~\ref{isomleft} and the previously established fact that $(\gamma_0,\epsilon_0,\phi_0)\neq (0,0,0)$,  $\mathfrak{K}(\nabla_0)$ is the Lie algebra of right-invariant vector fields (this follows also from Proposition~\ref{typei-ii}). In consequence, $f$  should map the complete vector field $A_0$ to a complete vector field of the form~$A_0+\lambda B_0$. But such a vector field cannot preserve~$V_0$ (be complete in restriction to~$V_0$), unless~$\lambda=0$. We must conclude that~$f_* A_0 =A_0$. Since~$B_0$ generates the derived algebra of~$\mathfrak{K}(\nabla_0)$, $f$ must map~$B_0$ to a multiple of it, positive if~$f$ preserves orientation, negative, if it does not. 

If~$f$ preserves orientation, up to composing with the flow of~$A$, we may suppose that~$f_*B_0=B_0$. By the orientation-preserving hypothesis and Proposition~\ref{isomleft}, $f$ is a left translation. But the fact that it preserves the markings of~$\nabla_0$ implies that~$f$ must be the identity. This proves that the orientation-preserving isometries of~$(V_0,\nabla_0)$ embed into the flow of~$A$. If~$f$ does not preserve orientation, it must exchange the special markings and we have~$n=m$. The mapping~$\beta$ of formula~(\ref{voltear}) is such an isometry and, after composing with it, we get an orientation-preserving isometry that embeds into the flow of~$A$.

Let~$f:M\to M$ be an orientation-preserving isometry. Since~$\mathfrak{K}(\nabla)$ has rank two exactly in the union of the~$V_i$, we should have~$f(V_0)=V_k$ for some~$k\in\mathbf{Z}$. The mapping~$\pi_{-k}\circ f|_{V_0}:V_0\to V_0$ is an isometry, orientation preserving if~$k$ is even, orientation-reversing if~$k$ is odd. If~$k$ is even (in particular, if~$m\neq n$), up to composing~$f$ with a suitable time of the flow of~$A$ on~$M$, we may suppose that~$\pi_{-k}\circ f|_{V_0}$ is the identity: this implies that~$f=\phi^k$. If~$k$ is odd,  composition with~$\phi^\frac{1}{2}$ reduces the claim to the previous case. \end{proof}

Let $G$ be a discrete  group  acting  by orientation-preserving isometries on $M$,  such that the quotient is a compact surface~$S$. Since every isometry preserves the nowhere vanishing vector field~$A$, $S$ admits a nowhere vanishing vector field and is thus a torus. Thus, in order to understand the compact quotients of~$M$, we should understand the homomorphisms~$\Phi:\mathbf{Z}^2\to \mathrm{Isom}^+(M,\nabla_{n,m})$. We already established an isomorphism between~$\mathrm{Isom}^+(M,\nabla_{n,m})$ and~$\mathbf{Z}\times\mathbf{R}$.

We claim that we can choose an ordered set of generators~$(e_1,e_2)$ of~$\mathbf{Z}^2$ such that
$\Phi(e_1)=(0,\tau)$ and $\Phi(e_2)=(k,\theta)$, with~$\tau>0$ and~$k>0$. Let~$\Phi(e_i)=(n_i,\tau_i)\in\mathbf{Z}\times\mathbf{R}$. We cannot have~$n_1=0$ and~$n_2=0$. Suppose first  that~$n_1n_2\neq 0$. Let~$k>0$ be the greatest common divisor of~$n_1$ and~$n_2$, so that~$n_i=k p_i$, with $p_1, p_2\in\mathbf{Z}$. Let~$q_1,q_2\in\mathbf{Z}$, such that~$p_1q_1+p_2q_2=1$. In the basis~$(p_2e_1-p_1e_2$, $q_1e_1+q_2e_2)$ of $\mathbf{Z}^2$, the images of the generators read~$(0,\tau)$ and~$(k,\theta)$, which reduces the case~$n_1n_2\neq 0$ to the case~$n_1n_2=0$. Up to replacing the generators by their inverses, we have~$\tau>0$ and~$k>0$. In this setting, $\tau$, $k$ and the class of $\theta$ in~$\mathbf{R}/\tau\mathbf{Z}$ are intrinsically defined and depend only upon the representation.  This proves item~(\ref{main-mod}) of Theorem~\ref{thm-global}.

\section{Quasihomogeneous connections on compact surfaces}\label{Quasi-compact-surfaces}

Let~$S$ be a compact oriented analytic surface endowed with a quasihomogeneous analytic connection~$\nabla_s$. Let~$\Omega\subset S$ be the open set  where the rank of the  Killing algebra is at least two. Our main assumption is that~$\Omega$ is neither empty,  nor all of~$S$. At the points in the boundary of~$\Omega$, $\nabla_s$ is locally isomorphic to one of the quasihomogeneous germs of connection classified in Theorem~\ref{thm-local}. Thus, the Killing algebra of~$\nabla_s$ is  isomorphic  either to~$\mathfrak{aff}(\mathbf{R})$, or to~$\mathfrak{sl}(2,\mathbf{R})$. We shall analyze separately these two cases in order to completely classify the pairs $(S, \nabla_s)$. In particular, we will prove, in Section~\ref{sl2}, that the case~$\mathfrak{sl}(2,\mathbf{R})$ does not occur.  Section~\ref{globalaff} achieves the proof of Theorem~\ref{thm-global}, in the case where the Killing algebra is isomorphic to~$\mathfrak{aff}(\mathbf{R})$.

\subsection{When the killing algebra is isomorphic to~$\mathfrak{aff}(\mathbf{R})$}\label{globalaff}

We will now prove part~(\ref{main-rec}) of Theorem~\ref{thm-global}. We will prove that, under the hypothesis that the killing algebra is isomorphic to~$\mathfrak{aff}(\mathbf{R})$, that $(S,\nabla_s)$ is isometric to a compact quotient of one of the affine manifolds~$(M,\nabla_{n,m})$ defined in the previous section.\\

 By Theorem~\ref{thm-local}, at the  points in the boundary of~$\Omega$, $\nabla_s$  is of Type~$\mathrm{I}$, $\mathrm{II}^0$ or $\mathrm{II}^1$.\\

Our first claim is that~\emph{$S$ is a torus and that each connected component of~$S\setminus\Omega$ is a cylinder}. We will begin by defining a non-singular foliation~$\mathcal{F}$ in~$S$. 
In the neighborhood of every point we may find two Killing vector fields~$A$ and~$B$ of~$\nabla_s$ such that~$[A,B]=B$. The vector field~$B$, the generator of the derived Lie algebra of~$\mathfrak{K}(\nabla_s)$ is well-defined up to multiplication by a constant. At the points where~$B$ does not vanish, define~$\mathcal{F}$ as the foliation tangent to~$B$. At the points where~$B$ vanishes, $B$ is, in suitable coordinates, a nonzero multiple of~$x\indel{y}$ (proof of Theorem~\ref{thm-local}; see Table~\ref{table1}). In this case, define~$\mathcal{F}$ locally as the kernel of~$dx$. This gives a foliation without singularities in all of~$S$, which is thus a torus.

Let~$p$ be a point in the boundary of~$\Omega$. If~$B$ does not vanish at~$p$, then the orbit of~$B$ through~$p$ (a leaf of~$\mathcal{F}$) gives, locally, the boundary of~$\Omega$. If~$B$ vanishes at~$p$, in the coordinates of Theorem~\ref{thm-local}, $\Omega$ is the complement of~$\{x=0\}$ and~$\mathcal{F}$ is given by the kernel of~$dx$. In all cases, the complement of~$\Omega$ is a union of leaves of~$\mathcal{F}$ and is thus   a finite union  of circles. By the Poincar\'e-Hopf index Theorem, every connected component of~$\Omega$ is a cylinder, as we claimed.\\

Thus, $S\setminus\Omega $ is a finite number of homologous simple closed curves. Let~$C$ be a connected component of~$\Omega$. We will denote by~$\overline{C}$ the manifold-with-boundary obtained by adding to each end of~$C$ the circle in~$S\setminus\Omega$ that compactifies it (this is, even if $S\setminus\Omega$ has only one connected component, $\overline{C}$ will still be considered as a manifold with two boundary components). Let~$\partial C$ be a boundary component of~$\overline{C}$. Let~$W$ be a tubular neighborhood of~$\partial C$ within~$S$. Let~$\Pi:\widetilde{C}\to C$ and~$P:\widetilde{W}\to W$ be universal coverings of~$C$ and~$W$. Since the inclusions of~$C\cap W$ in~$C$ and in~$W$ are homotopy equivalences, there is a natural identification between~$\Pi^{-1}(C\cap W)$ and~$P^{-1}(C\cap W)$. 

The restriction of~$\nabla_s$ to~$C$ is locally homogeneous at every point. By Proposition~\ref{constgx}, Remark~\ref{rem-gxorient} and item~(\ref{isomleft-ori}) of Proposition~\ref{isomleft}, there exists a left-invariant connection~$\nabla_0$ on~$\mathrm{Aff}_0(\mathbf{R})$, a \emph{developing map}~$\mathcal{D}:\widetilde{C}\to\mathrm{Aff}_0(\mathbf{R})$, which is a local  isomorphism  between~$\Pi^{-1}(\nabla_s)$ and~$\nabla_0$, and a \emph{holonomy} morphism~$\mu:\pi_1(C)\to \mathrm{Aff}_0(\mathbf{R})$ such that for every~$q\in\widetilde{C}$ and~$\gamma\in\pi_1(C)$,
\begin{equation}\label{developing}\mathcal{D}(\gamma \cdot q)=\mu(\gamma)\mathcal{D}(q),\end{equation}
where~$\gamma \cdot q$ is the image of~$q$ under the deck transformation associated to~$\gamma$. \\

We now claim that \emph{the image of the holonomy is non-trivial and belongs to a semisimple one-parameter group}. Let us prove this. If the holonomy  is trivial, the developing map is, by formula~(\ref{developing}), invariant by the action of~$\pi_1(C)$ by deck transformations and, in consequence, induces a well-defined map $\mathcal{D}^\flat:C\to\mathrm{Aff}_0(\mathbf{R})$. The Killing vector fields of~$\nabla_s$ are then globally well-defined on $C$: they are  pull-back of  right-invariant vector fields on~$\mathrm{Aff}_0(\mathbf{R})$. The Killing vector fields of~$\nabla_s$ in~$C$ are complete, for they preserve the boundary, and hence induce a (transitive)  action of~$\mathrm{Aff}_0(\mathbf{R})$ on~$C$. The map~$\mathcal{D}^\flat$ is equivariant with respect to this action of~$\mathrm{Aff}_0(\mathbf{R})$ on~$C$ and its action by left translations upon itself. But the first action has a non-trivial stabilizer and the second one is free. This contradiction shows that the holonomy cannot be trivial.

Assume, by contradiction, that  the image of the holonomy is generated by an element of the form~$(1,\tau)$ with~$\tau\neq 0$. This is the flow of~$B_0$ in time~$\tau$. In particular, $B_0$ is holonomy-invariant and induces a well-defined Killing vector field~$B_s$ in~$C$, which preserves the  boundary  and is thus complete. Furthermore, its solutions are all periodic and have  period~$\tau$ (it is \emph{isochronous}). The developing map~$\mathcal{D}$ maps~$\Pi^{-1}_*(B_s)$ to a constant multiple of~$B_0=\partial/\partial v$. The completeness of~$B_s$ in~$C$ implies that the image of the developing map~$\mathcal{D}(\widetilde{C})$ is saturated by~$B_0$, and is thus of the form $V=\{u_0<u<u_1\}$, with~$0\leq u_0<u_1\leq\infty$. 

Let~$p\in \partial C$. By Theorem~\ref{thm-local}, there exists a neighborhood~$U$ of~$p$ and a mapping~$\Psi_0:(U,p)\to(\mathbf{R}^2,O)$ that maps~$\nabla_s$ to a connection~$\nabla$ of Type~$\mathrm{I}$ or of Type~$\mathrm{II}$. The image of~$B_s$ under this mapping is a multiple of the vector field~$B$, $\Psi_0(B_s)=\lambda B$ for some~$\lambda\in\mathbf{R}$. (Notice that~$\lambda B$ is a complete vector field in~$\mathbf{R}^2$.) By its isochronicity, $B_s$ cannot vanish at~$p$ and thus~$B$ does not vanish at~$O$: the connection~$\nabla$ cannot be of Type~$\mathrm{II}$ and is necessarily of Type~$\mathrm{I}$ (by Prop.~\ref{propaffine}).

Suppose, without loss of generality, that~$W$ (the tubular neighborhood of~$\partial C$) is saturated by the flow of~$B_s$. Let~$T\subset \widetilde{W}$ be a transversal intersecting once and only once each orbit of~$P^{-1}_*B_s$ and such that~$P(T)\subset U$.  For each~$q\in \widetilde{W}$, let~$t_q\in\mathbf{R}$ be such that~$\Phi_{P_*^{-1}B_s}^{t_q}(q)$, the flow of~$P_*^{-1}B_s$ in time~$t_q$ applied to~$q$, belongs to~$T$. Let~$\Psi:\widetilde{W}\to\mathbf{R}^2$ be the mapping defined by
\begin{equation}\label{analogue}\Psi(q)=\Phi_{\lambda B}^{-t_q}\circ\Psi_0\circ P \circ\Phi_{P_*^{-1}B_s}^{t_q}(q).\end{equation}
It is a well-defined  diffeomorphism onto its image which  maps~$P^{-1}(\nabla_s)$ to~$\nabla$ 
and $P^{-1}_*B_s$ to~$\lambda B$ (it is, in some sense, a developing map).

The vector field $\lambda B$ has the first integral~$h(x,y)=x$ and $(X\cdot h)/h=-1/n<0$ (see Table~\ref{table1}). The function~$\Psi^*h$ and the vector field~$\Psi^{-1}_*X$ are defined in~$\widetilde{W}$. They are invariant under the flow of~$P^{-1}(B_s)$. In consequence, their images $h_s=(\Psi\circ P^{-1})^*h$ and~$X_s=P_*\Psi^{-1}_*X$  are well-defined in~$W$ and satisfy the same relation.

Hence, for each point~$q$ in~$C\cap W$, the orbit of $X_s$ starting at~$q$ is defined for all positive time. Moreover, this orbit is properly embedded in~$C$ and, as time goes to infinity, accumulates to~$\partial C$, the zero set of~$h_s$.

The image of~$\Pi^{-1}_*X_s$ under~$\mathcal{D}$ is a vector field of the form~$X_0+\lambda Y_0=u\partial/\partial u+[\cdots]\partial/\partial v$. By our previous arguments, for the initial conditions coming from~$C\cap W$, the forward orbits of this vector field are complete in~$V$. We conclude that the constant~$u_1$ defining $V$ is infinite and that~$p$ belongs to the end of~$C$ corresponding to~$u_1$. The  normal forms do not allow us to describe the end of~$C$ corresponding to~$u_0$, this is, such an end is impossible to compactify. This shows, by contradiction,  that the holonomy is not of the form~$(1,\tau)$.

We conclude that the holonomy is generated by an element of the form~$(e^{\tau},v)$ with~$\tau\neq 0$. It is conjugated to~$(e^{\tau},0)$, the flow of~$A_0$ in time~$\tau$.  This proves our claim. \\

We will henceforth suppose that the holonomy is generated by the flow of~$A_0$ in time~$\tau$. The holonomy preserves $A_0$ and  induces, via~$\mathcal{D}$, a well-defined Killing vector field~$A_s$ in~$C$. As in the  previous case, $A_s$ will be complete and isochronous with period~$\tau$.  The image of the developing map~$\mathcal{D}(\widetilde{C})$, being saturated by~$A_0$, is a cone of the form~$V=\{c_0<v/u<c_1\}$,
with~$-\infty\leq c_0<c_1\leq\infty$. There is a one-to-one mapping~$\mathcal{D}^\flat:C\to\langle \Phi_{A_0}^\tau \rangle\backslash  V$.

Let~$\partial C$ be a boundary component of~$C$ and let~$p\in\partial C$. The vector field~$A_s$ does not vanish at~$p$. Let~$B_s\in\mathfrak{K}(\nabla_s)$ be a Killing vector field in a neighborhood of~$p$  satisfying~$[A_s,B_s]=B_s$ (it is well-defined up to multiplication by a constant). Either $B_s$ vanishes identically along~$\partial  C$ or it is everywhere non-zero (by Prop.~\ref{propaffine}, the two cases lead to different normal forms for $\nabla_s$ at $p$).
\begin{itemize}
\item \emph{If~$B_s$ vanishes identically along~$\partial C$}, then  there exists~$n\in\frac{1}{2}\mathbf{Z}$, parameters~$(\gamma,\epsilon,\phi)$ and a map~$\Psi_0:(U,p)\to(\mathbf{R}^2,(0,1))$ mapping~$\nabla_s$ to the  connection~$\nabla$ of Type~$\mathrm{II}^1(n)$, with parameters~$(\gamma,\epsilon,\phi)$, and mapping~$C\cap U$ to~$\{x>0\}$ (Theorem~\ref{thm-local} and Remark~\ref{twoside}). Let~$E$ be the vector field in~$\mathbf{R}^2$ that is mapped by~$\mathcal{D}\circ\Psi_0^{-1}$ to~$A_0$. It is of the form~$A+\lambda B$, for some $\lambda \in \mathbf{R}$. By post-composition of $\Psi_0$ with a suitable time of the flow of~$B$, we will suppose that~$E=A$. The vector field $E$ has the primitive first integral~$h(x,y)=xy^{1/(n-1)}$ and $X\cdot \log(h)=1/(1-n)<0$.

\item \emph{If~$B_s$ does not vanish identically along~$\partial C$}, then there exists~$n\in\frac{1}{2}\mathbf{Z}$, parameters~$(\gamma,\epsilon,\phi)$, and a map~$\Psi_0:(U,p)\to(\mathbf{R}^2,0)$ mapping~$\nabla_s$ to the  connection~$\nabla$ of Type~$\mathrm{I}(n)$, with parameters~$(\gamma,\epsilon,\phi)$ and mapping~$C\cap U$ to~$\{x>0\}$. Let~$E=A+\lambda B$ ($\lambda\neq 0$) be the vector field that maps into~$A_0$ via~$\mathcal{D}\circ \Psi_0^{-1}$. Up to a change of coordinates given by the flow of~$A$ and since~$\rho_*B=-B$ for~(\ref{miroirtau}), we may suppose that~$\lambda=-1$, this is, $E=A-B$. It has the primitive first integral~$h(x,y)=x(1-y)^{1/n}$ that satisfies $X\cdot \log(h)=-1/n<0$. 
\end{itemize}

Define~$\Psi:\widetilde{W}\to\mathbf{R}^2$ in a way analogue to~(\ref{analogue}). It is a diffeomorphism onto its image, maps~$P^{-1}(\nabla_s)$ to~$\nabla$ and $P^{-1}_* A_s $ to~$E$. Moreover, it is equivariant with respect to the flow of~$P^{-1}_* A_s$ in~$\widetilde{W}$ and the flow of~$E$ in~$\Psi(\widetilde{W})\subset\mathbf{R}^2$. Hence, if~$\Phi_E^\tau:\mathbf{R}^2\to\mathbf{R}^2$ denotes the flow of~$E$ in time~$\tau$, there is a one-to-one mapping~$\Psi^\flat:W\to \Psi(\widetilde{W})/\langle\Phi_E^\tau\rangle$.

As before, the function $h_s=(\Psi\circ P^{-1})^*h$ and the vector field~$X_s=P_*\Psi^{-1}_*X$ are well-defined in~$W$ and~$X_s\cdot \log(h_s)$ is a strictly negative real. The level curves of~$h_s$ are circles. Those in~$C$ are  orbits of~$A_s$ and~$\partial C=h_s^{-1}(0)$. For each point~$q$ in~$C\cap W$, the orbit of $X_s$ is, in positive time, \emph{complete  and properly embedded}  in~$C$.

Suppose that the end of~$C$ associated to~$\partial C$ corresponds to the end of~$V$ associated to~$c_1$.  The image of~$\Pi^{-1}_*X_s$ under~$\mathcal{D}$ is a vector field of the form~$Z=X_0+\lambda Y_0=u(\partial/\partial u-\lambda \partial/\partial v)$. For the forward orbits of~$Z$ to be complete and properly embedded in~$\langle \Phi_{A_0}^\tau \rangle\backslash V$, we must have that~$V$ is, on the corresponding side, bounded by~$v/u=-\lambda$, the common orbit of~$Z$ and~$A_0$ (otherwise, $Z$ will not be complete in~$\langle \Phi_{A_0}^\tau \rangle\backslash V$ or its orbits will not be properly embedded: see Figure~\ref{fig:propemb}). In particular, $c_1<\infty$.\\

\begin{figure}
 \includegraphics[width=3in]{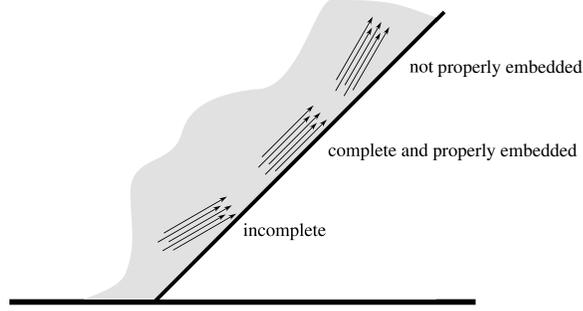}
\caption{In the quotient of the cone~$V$, the forward orbits of a semisimple left-invariant field may be incomplete, may not be properly embedded or, in only one case, may be complete and properly embedded.}\label{fig:propemb}
\end{figure}

Thus, the connection~$\nabla_0$ has a marking~$Z$ and~$V$ is bounded by the unique common orbit of~$Z$ and~$A_0$. Since~$Z$ is defined in all of~$V$,  the vector field~$X_s$ may be extended to all of~$C$. Every orbit of~$X_s$ in~$C$ accumulates to~$\partial C$ as time goes to infinity. 

For the other boundary component~$\partial'C$ of~$\overline{C}$ we define analogue  objects  $X_s'$, $h_s'$, $Z'$, $n'$, $\nabla_0'$, $(\gamma',\epsilon',\phi')$. Notice that~$X_s$ and~$X_s'$ are different, since their orbits in positive time converge to different boundary components of~$\overline{C}$. \\

In the neighborhood of a point~$p$ in~$\partial C$, $\Psi_0(X_s)=X$ and~$\Psi_0(X_s')=X+\lambda Y$ for some~$\lambda\in\mathbf{R}$. By Remark~\ref{deuxtypes}, the nature of~$\nabla_s$ at~$\partial' C$ (if it is either of Type~$\mathrm{I}$ or of Type~$\mathrm{II}$) and the value of~$n'$ determine the value of~$(\gamma,\epsilon,\phi)$ in~$\nabla$, up to the natural equivalence. In particular, by (the proof of) Lemma~\ref{nonvanishingparams}, $\gamma$ and~$\phi$ cannot vanish simultaneously. We must conclude that~$n\in\mathbf{Z}$. In an analogue way, $n'\in\mathbf{Z}$.\\

We now have that~$\overline{C}$ is very similar to the quotient of~$V_0$ (as defined in the previous section) under the flow of~$A$ in time~$\tau$, except for the assumption upon the parity of~$n$ and~$n'$. We will now prove that \emph{if~$n$ is even (resp. if~$n$ is odd), $\nabla_0$ is of Type~$\mathrm{I}$ (resp. of Type~$\mathrm{II}$).}\\

Let us consider the cylinder~$C'$ neighboring~$C$ across~$\partial C$. It has the boundary components~$\partial C$ and~$\partial''C$. Consider the coordinate~$\Psi_0$ around the point~$p\in\partial  C$ and suppose (Remark~\ref{twoside}) that this coordinates map~$C \cap U$ to~$\{x>0\}$. In~$\{x>0\}$ we have two vector fields~$X$ and~$X+\lambda Y$ that give two special (Definition~\ref{defnmarking}) markings for the corresponding connection in the affine plane. These two vector fields induce also two special markings for the connection in the affine plane induced by the restriction of~$\nabla_s$ to~$C'$. However, an invariant connection on the affine group  does not have three special markings:  indeed, by contradiction, from formula~(\ref{nothreespmarkings}),
\begin{itemize}
\item if there are two markings of Type~$\mathrm{I}$, $\delta=1$ and~$\delta'=1$, and thus $\alpha''=\delta''$: the third special marking must be of Type~$\mathrm{II}$ (for~$\delta''\neq\alpha''+1$), and thus  $n=-1$. This is impossible: we conclude that the connection does not have three special markings;
\item if there are two markings of Type~$\mathrm{II}$, $\delta=1+\alpha$ and~$\delta'=1+\alpha$ and thus $3\delta''=\alpha''+1$. If the third marking is of Type~$\mathrm{II}$, then $\delta''=0$ and~$\alpha''=-1$; if it is of Type~$\mathrm{I}$,  then $\delta''=1$ and~$\alpha''=2$. In both cases this is in contradiction with the possible values of $n$. We must conclude that the connection does not have three special markings.
\end{itemize}
We conclude that \emph{the vector field~$X+\lambda  Y$ induces both the special marking corresponding to~$\partial' C$ and the one corresponding to~$\partial'' C$}. Hence, in~$C$ and in~$C'$, the special markings are given by~$X$ and~$X+\lambda Y$.  For the connections of Type~$\mathrm{II}^1$, we have
$$\frac{E\wedge (X+\lambda Y)}{E\wedge X}=1+\lambda \frac{E\wedge Y}{E\wedge X}=1+\frac{\lambda}{x^{n-1}y}, $$
whose sign, close to~$(0,1)$, is determined by the sign of~$\lambda/x^{n-1}$. For connections of Type~$\mathrm{I}$, we have
$$\frac{E\wedge (X+\lambda Y)}{E\wedge X}=1+\lambda \frac{E\wedge Y}{E\wedge X}=1+\frac{\lambda}{x^n(y+1)}. $$
Its sign is controlled by~$\lambda/x^n$.  But we previously established in~(\ref{orient}) that these expressions must be negative. Hence, $n$ must be an even integer for Type~$\mathrm{I}$ and an odd one for Type~$\mathrm{II}^1$, as claimed.\\

This implies that~$\overline{C}$ is covered by~$\overline{V}_0$ for the connection~$\nabla_{n,n'}$ of the previous section and that the same is true of~$\overline{C'}$, since the connection in the universal covering is determined by the numbers attached to the special markings and since the period of~$A_s$ is the same in both~$C$ and~$C'$. This implies that~$(S,\nabla_s)$ is a global quotient of~$(M,\nabla_{n,n'})$. This finishes the proof of Theorem~\ref{thm-global} in the case where~$\mathfrak{K}(\nabla)\approx\mathfrak{aff}(\mathbf{R})$.

\subsection{When the killing algebra is isomorphic to~$\mathfrak{sl}(2,\mathbf{R})$}   \label{sl2} We prove here:

\begin{proposition} Let $S$ be a compact real analytic surface endowed with a  real-analytic torsion-free affine connection $\nabla_s$. If the Killing algebra of~$\nabla_s$ is isomorphic to  $\mathfrak{sl}(2,\mathbf{R})$ and $\nabla_s$ is locally homogeneous on a nontrivial open set in $S$, then $\nabla_s$ is locally homogeneous on all of $S$.
\end{proposition}

Let~$\Omega\subset S$ be the maximal  open set where there Killing algebra is transitive. The normal form of $\nabla_s$ at points of the boundary of~$\Omega$ is that of  Type~$\mathrm{III}$. In particular, $\mathfrak{K}(\nabla_s)\approx\mathfrak{sl}(2,\mathbf{R})$ and the points of $S \setminus \Omega$ are isolated ($\Omega$ is connected). The restriction of~$\nabla_s$ to~$\Omega$ is locally homogeneous everywhere and is locally modeled in the homogeneous space of~$\widetilde{\mathrm{SL}(2,\mathbf{R})}$ (the universal covering  of $\mathrm{SL}(2, \mathbf{R})$) associated to the subgroup generated by~$\left(\begin{array}{cr} 0 & 1 \\ 0 & 0 \end{array}\right)\in\mathfrak{sl}(2,\mathbf{R})$. This homogeneous space is, naturally, the universal covering of~$\mathbf{R}^2\setminus\{0\}$. Notice that it is also acted on  transitively by $\widetilde{\mathrm{GL}(2,\mathbf{R})}$.

Let us  define a foliation with singularities~$\mathcal{F}$ on~$S$. For each point of~$\Omega$, define~$\mathcal{F}$ as the foliation tangent to the centralizer of~$\mathfrak{K}(\nabla_s)$ in the Lie algebra
generated by the $\widetilde{\mathrm{GL}(2,\mathbf{R})}$-action. In~$\mathbf{R}^2\setminus\{0\}$, this centralizer is generated by~$x\indel{x}+y\indel{y}$. For the points~$p\notin\Omega$, in the coordinates of Theorem~\ref{thm-local}, the foliation extends as the foliation with singularities induced by the previous vector field at~$0$.  These are the only singularities of~$\mathcal{F}$. They have Poincar\'e-Hopf index~$1$. Hence, $S$ is a sphere, the complement of~$\Omega$ consists of two points and~$\Omega$ is a cylinder. 

Since~$S$ is simply connected, all the killing vector fields are globally well-defined in~$S$ and induce an action of~$\widetilde{\mathrm{SL}(2,\mathbf{R})}$ that, in fact (from the normal forms in the neighborhood of the points that do not belong to~$\Omega$), factors through a faithful action of~$\mathrm{SL}(2,\mathbf{R})$. By the normal forms, the restriction of this last action to~$\Omega$ is transitive.
Thus $\Omega$ identifies with $\mathbf{R}^2\setminus\{0\}$, seen as a homogeneous space of $\mathrm{SL}(2,\mathbf{R})$. 

We can  see~$S$ as the identification of two copies of~$\mathbf{R}^2$ glued along~$\mathbf{R}^2\setminus\{0\}$ in a $\mathrm{SL}(2,\mathbf{R})$-equivariant way. In particular, there is an invariant area form on~$S$ (which is a constant multiple of the standard area form  in restriction to each copy of~$\mathbf{R}^2$). Since $S$ is compact, the integral of the area form is finite. In the copies of~$\mathbf{R}^2$, the integral of the area form is infinite. This contradiction finishes the proof of the Proposition and of Theorem~\ref{thm-global}.

\section{Connections in~$\mathbf{R}^2$ that are invariant under the special linear group}   \label{sl2chapter}

We will now study the geodesics of the connection of Type~III in~$\mathbf{R}^2$. We will prove that such a connection is geodesically complete if and only if~$\gamma=0$. We may describe the connections and their geodesics by making use of the \emph{special affine curvature} of curves in the plane (see, for example, \cite[Chapter~I]{spivak}):

\begin{proposition} A non-flat torsion-free  real-analytic  geodesically complete connection $\nabla$  in~$\mathbf{R}^2$, invariant under the action of~$\mathrm{SL}(2,\mathbf{R})$, is of the form
$$\nabla_{\del{x}}\del{x}=ky^2E,\;
 \nabla_{\del{x}}\del{y}=-kxyE,\;
 \nabla_{\del{y}}\del{y}=kx^2E,$$
for~$E=x\indel{x}+y\indel{y}$ and some~$k\in\mathbf{R}^*$. The parametrized geodesics are either:
\begin{itemize}
 \item  lines through the origin at constant speed  or
\item conics of special affine curvature~$\sqrt[3]{k}$, centered at the origin and parametrized at constant speed with respect to special affine arc length.
\end{itemize}
\end{proposition}

A straightforward calculation shows that the curve~$s\mapsto (s,0)$ is a geodesic of every connection of Type~III and thus every line through the origin parametrized at constant speed is a geodesic. Let~$v(s)=(x(s),y(s))$ be a non-constant geodesic of~$\nabla$ that is not a line through the origin, this is, such  that $v(s)$ and~$v'(s)$ are linearly independent. For each~$s$, let~$A(s)\in\mathrm{SL}(2,\mathbf{R})$ such that~$A(s)v(s)=(1,0)$ and such that~$A(s)v'(s)=(0,u(s))$. We have~$u =x y' -y x'$. The equation of the geodesics becomes the first-order non-linear equation
$u'=2\gamma \tau^2$. The solution to this equation with initial condition~$\tau_0$ is
$$u(s)=\frac{u_0}{1-2\gamma u_0 s}.$$
We must conclude that, if the connection is geodesically complete, $\gamma=0$. In this case, if~$\epsilon=-2$, then a  geodesic is given by~$v(s)=(\cos(s),\sin(s))$. By invariance, the other geodesics are ellipses of area~$1$ centered at the origin (equivalently, ellipses centered at the origin whose special affine curvature is~$1$). If~$\epsilon=2$, then a  geodesic is given by~$v(s)=(\cosh(s),\sinh(s))$ (a hyperbola centered at the origin with special affine curvature equal to~$1$). By the natural rescalings induced by homotheties, we obtain the proof of the Proposition.

\bibliography{lesreferences}
\bibliographystyle{alpha}

\end{document}